\documentclass[preprint,12pt]{elsarticle}

\usepackage[margin=2.0cm]{geometry}

\usepackage{graphicx}
\usepackage{amssymb}
\usepackage{amsmath}
\usepackage{siunitx}
\usepackage[dvipsnames]{xcolor}
\usepackage{commath}
\usepackage{amsfonts}%
\usepackage{mathrsfs}
\usepackage{graphicx,verbatim}

\usepackage{comment}
\usepackage[normalem]{ulem}

\usepackage{caption}
\usepackage{subcaption}

\usepackage{tikz}
\usetikzlibrary{cd}
\usetikzlibrary{trees}
\usetikzlibrary{calc}

\usepackage{tcolorbox}
\usepackage{colortbl}
\usepackage[pagebackref=false,colorlinks=true, pdfstartview=FitV, linkcolor=violet,citecolor=red, urlcolor=blue]{hyperref}

\usepackage{extarrows}

\usepackage{amsthm}
\newtheorem{lthm}{Theorem}

\newtheorem{ldef}{Definition}

\usepackage{dirtytalk}

\usepackage{hyperref}

\def\Z{\ensuremath {{\mathbb Z}}}

\def\cO{\ensuremath {{\mathcal O}}}

\def\cI{\ensuremath {{\mathcal I}}}

\def\scB{\ensuremath {{\mathscr{B}}}}

\def\scP{\ensuremath {{\mathscr{P}}}}
\def\scR{\ensuremath {{\mathscr{R}}}}
\def\scH{\ensuremath {{\mathscr{H}}}}
\def\scT{\ensuremath {{\mathscr{T}}}}

\def\wzeta{\ensuremath {\widetilde{\zeta}}}

\DeclareMathOperator{\val}{val}
\DeclareMathOperator{\valp}{\val_p}
\DeclareMathOperator{\valpi}{\val_{\pi}}

\DeclareMathOperator{\GL}{\mathrm{GL}}
\DeclareMathOperator{\SL}{\mathrm{SL}}

\newcommand{\defbf}[1]{\textbf{#1}}

\usepackage[font=small,labelfont=bf]{caption}

\newtheoremstyle{upright}
  {6pt}{6pt}
  {\normalfont}
  {}{\bfseries}
  {.}{.5em}{}

\newtheorem{theorem}{Theorem}
\newtheorem*{theorem*}{Theorem}
\newtheorem{definition}{Definition}
\newtheorem{proposition}{Proposition}
\newtheorem{corollary}{Corollary}
\newtheorem{lemma}{Lemma}

\newtheorem*{conjecture*}{Conjecture}

\newtheorem{remark}{Remark}

\theoremstyle{upright}
\newtheorem{examplex}{Example}
\newenvironment{example}
  {\begin{examplex}}
  {\hfill$\square$\end{examplex}}

\usepackage{lineno}

\journal{arXiv}

\begin{document}

\begin{frontmatter}


\title{Impacted Buildings for $\GL(2)$}



\author{Malors Espinosa}

\author{Zander Karaganis}


\begin{abstract}

In this paper we define a generating function for buildings of type $\widetilde{A}_1$ (i.e. trees) that are enhanced with a certain filtration structure. We prove that this generating function recovers the zeta function of certain quadratic orders. We do this by studying how the ideals of the orders distribute in the building of $SL(2, K)$.

\end{abstract}
\end{frontmatter}

\section{Introduction}\label{sec: Introduction}

\subsection{The goal of this article}

This paper is concerned with the way that certain zeta functions can be defined and computed. Let $K$ be a $p$-adic field and $\cO_K$ be its ring of integers. Let $L$ be a reduced $K$-algebra of dimension $2$ over $K$ and $\cO_L$ the integral closure of $\cO_K$ in $L$. In \cite{ZYun}, Yun defines for a given order $\cO\subseteq L$ a zeta function $\zeta_{\cO}$ and proves it is a rational function
\begin{equation*}
    \zeta_{\cO}(s) = \dfrac{P(q^{-s})}{V(q^{-s})},
\end{equation*}
where $q$ is the cardinality of the residue field of $K$. The polynomial $V$ is determined by $L/K$, while the polynomial $P$ is determined by $\cO$. 

Let $E$ be a quadratic extension of a number field $F$. For this extension, Langlands in \cite{LanBE04} and Altu\u{g} in \cite{AliI}, construct a formula for the evaluation of the product of certain local orbital integrals. This formula is essential for the success of the analysis being carried out. However, in the aforementioned papers, its source is \textit{ad hoc}. Arthur, in \cite{ArtStrat}, predicted that the polynomials $P$, as $K$ varies on the localization at different primes of $E$ (and $L$ being the appropriate extension according to the splitting behavior of the prime), recover the Euler factors of the formulas. Consequently, it became of interest to compute these polynomials for the quadratic case. They are computed in \cite{MEELZETA}. Then, in \cite{MEELMULT}, the explicit polynomials are used to recover the formulas, thus confirming the prediction of Arthur. 

The above process leaves two problems to be discussed. On the one hand, in \cite{ZYun} it is proven that these order zeta functions exists and continue to be rational for $L/K$ of dimension $n$. Furthermore, the corresponding polynomials $P$ are also expected to be the local factors for formulas that can play the analogous role as those of Langlands and Altu\u{g} did for the quadratic case. The method in \cite{MEELZETA} runs into complications when dealing with $n\ge 3$.

To compute the zeta functions in \cite{MEELZETA}, one verifies that in the cases of interest, they are the classical ideal counting zeta functions. Then one proves it is enough to study the contribution of the principal ideals to the zeta function. In order to quantify this contribution, the principal ideals are classified in different classes, and then the cardinality of each class is explicitly computed. For higher dimension cases, both steps are unclear for it is not evident how to organize ideals. On the other hand, the importance of these formulas arises from their relationship with orbital integrals for $\GL(n)$. When this group changes, it is unclear what the analogous zeta functions could be. 

It is tempting to approach both of these issues through the Bruhat-Tits building alone, however, such an approach has yet to be employed.
The difficulty with this approach is that these zeta functions count ideals, as opposed to rank $2$ lattices. This obscures their distribution in the building and hence hinders their counting. Nevertheless, to find such a possibility is desirable because such methods could be transferred to other groups and their Bruhat-Tits building.

The purpose of this article is to give a construction in the building of $\SL(2, K)$ (i.e. a tree) and then reinterpret and compute the zeta function $\zeta_{\cO}$ with these new definitions.

\subsection{Results and organization of this article}

Our first main definition is the following one.

\begin{ldef}
   An \defbf{impacted building} is a pair $(\scB, \scP)$, where $\scB$ is a building of type $\widetilde{A}_1$ and $\scP$ is a subcomplex of $\scB$ which we call the \defbf{basin}.  
\end{ldef}
Given an impacted building we define for $n\ge 1$,
\begin{equation*}
    \scP_n = \left\{v\in V(\scB) \mid d(x, \scP_{n-1})\le 1 \right\}.
\end{equation*}
Consequently, we have a layer structure around the basin $\scP$, by defining
\begin{equation*}
    \scR_n = \scP_n\backslash \scP_{n-1},
\end{equation*}
and $\scR_0 = \scP_0 = \scP$. 

Given a vertex $v$, we denote by  $h(v)$ as the unique $n$ for which $x\in\scR_n$. With this, we are able to provide our second main definition.
\begin{ldef}
    Let $(\scB, \scP)$ be an impacted building and let $v$ be any vertex of $\scB$. We define the \defbf{layer generating function from v} as
    \begin{equation*}
        \zeta_v^{\scR}(X) := \displaystyle\sum_{d = 0}^{\infty}\#\left\{x\in \scR_{h(v)} \mid x \mbox{ can be reached by a path of length } d \mbox{ from } v  \right\}X^d.
    \end{equation*}
\end{ldef}

The definition of impacted building shows that the building by itself is not enough to make the desired computations. A basin, and hence a particular structuring into layers, is also required. This reorganization is captured in the coefficients of the generating functions from the vertices. We will discuss these definitions and related ones, as well as some of their basic properties, in Section \ref{sec: Impacted Buildings and their Generating Functions}.

Once we are in this context, three particular choices for $\scP$ become important. A single vertex, a single edge or a whole apartment. We will call the corresponding impacted buildings the ramified, unramified and split impacted buildings and denote them $\scB_R, \scB_U, \scB_S$. In Section \ref{sec: Explicit computations of generating functions} we will compute the corresponding zeta functions for these three cases. As we will see, the computations are straightforward and do not require any knowledge from arithmetic origin. They belong to the impacted building alone. In particular, at this stage these computations recover the values of $\zeta_{\cO}$ for the quadratic case, even though we are not yet making the explicit connection between the two perspectives.

Section \ref{sec: Order polynomials for GL(2)  and their computations via ideal types} briefly discusses the results of \cite{MEELZETA}. We need these results to be able to establish the connections between the two perspectives. We keep the discussion as brief as possible and refer the reader to the original paper for proofs. In Section \ref{sec: Unit distribution and natural appearance of Impacted Buildings} we use the reviewed results to explicitly locate the ideals of the orders $\cO_n$ we care about (and define in Section \ref{sec: Order polynomials for GL(2)  and their computations via ideal types}). Concretely, we prove

\begin{lthm}
    Let $L/K$ be one of our three cases (i.e. ramified, unramified, or split), $n\ge 0$ and $\scB$ the corresponding impacted building (i.e. ramified, unramified, or split, respectively). A vertex $v$ of $\scB$ represents a (principal, rank $2$) ideal of $\cO_n$ if and only if $v\in\scR_n$. In here $\scB$ is the corresponding building.
\end{lthm}

This theorem shows that the impacted buildings that we have independently studied indeed contain among them those required for the study of the zeta functions of orders. We have mentioned above that the study of the zeta function succeeds because it can be reduced to that of the principal ideals. Instead, in the impacted building we have the basin generating functions. The theorem above is qualitatively establishing that the principal ideals are to the ideals what the layer $\scR_n$ is to $\scP_n$. 

In order to solidify that these are the same structure seen from two different views, we must prove that quantitatively they also correspond. That is, their contribution to the corresponding zeta functions coincide. That is our main theorem:

\begin{lthm}\label{thm: principal zeta = vertex zeta}
    Let $L/K$ be one of the above considered extensions (i.e. unramified, ramified, split extensions) and $\scB$ the corresponding impacted building (i.e. $\scB_U, \scB_R$, $\scB_S$). Then we have the equality
    \begin{equation*}
        \zeta_{\cO_n}^P(s) = \zeta_{\scB, \cO_n}(q^{-s}),
    \end{equation*}
    where $\zeta_{\scB, \cO_n}(X)$ is the generating function from the vertex associated to $\cO_n$ for the corresponding impacted building and $\zeta_{\cO_n}^P$ is the contribution of the principal ideals to $\zeta_{\cO_n}(s)$.
\end{lthm}

This theorem is proven in Section \ref{sec: Discussion and proof of the main theorem}. We have mentioned before that in \cite{MEELZETA} the contribution of the principal ideals is computable because one can classify them into different types and then count how many ideals of each type exist. The essential idea to prove the main theorem is to describe precisely which types are equivalent to a given vertex $v$ of $\scR_n$. We will prove that is codified by the distance $d(v, \cO_n)$. In this way, the types disappear from the picture and everything gets reorganized into the layer zeta functions.

In particular, the main theorem and the computations of Section \ref{sec: Explicit computations of generating functions}, give an alternative way to obtain the values of the polynomials $P$ we have mentioned before. The main difference lies in the fact that we \textit{do not} require to compute how many ideals are of each individual type. Instead, we count how many vertices are at a given distance within a given layer.
\subsection{Previous literature}

Similar zeta functions associated with orders were defined and studied by Solomon in \cite{Solomon1977ZetaFA}. Their properties were furthermore understood and developed in the important articles \cite{Reiner1}, \cite{Reiner3}, \cite{Reiner2} by Bushnell and Reiner. These zeta functions differ from those we discuss in this paper in that they count $\cO_K$-lattices of rank $2$ inside of $\cO_n$, while the ones defined by Yun in \cite{ZYun} count ideals. As we have mentioned, it is precisely this difference what leads to the impacted building to acquire the extra information not captured by the building alone.

On the other hand, the polynomials $P$ that we have mentioned, and which we recompute in Section \ref{sec: Explicit computations of generating functions} for the quadratic case,  were already computed in \cite{Kaneko2003} and more recently \cite{KMizuno} by different methods.

\section{Impacted Buildings and their Generating Functions}\label{sec: Impacted Buildings and their Generating Functions}

\subsection{The Generating Functions from a Vertex}

Let $\scB$ be a building with fundamental apartment of type $\widetilde{A}_1$; that is, the apartments are real lines with vertices at each integer, and the building is a tree.

We begin with our construction by introducing a simple but fundamental concept.
\begin{definition}\label{def: impacted building}
    We define an \defbf{impacted building} as a pair $(\scB, \scP)$, where $\scB$ is a building of type $\widetilde{A}_1$ and $\scP$ is a connected subcomplex of $\scB$ which we call the \defbf{basin}.
\end{definition}

Given an impacted building we construct a filtration
\begin{equation*}
    \scP = \scP_0 \subset \scP_1 \subset \scP_2 \subset \scP_3 \subset \cdots,
\end{equation*}
by declaring inductively for $n\ge 1$,
\begin{equation}\label{eqn: defining nth basin}
    \scP_n = \left\{v\in V(\scB) \mid d(x, \scP_{n-1})\le 1 \right\}.
\end{equation}
Associated with this filtration, we have an associated layer structure around the basin $\scP$, by defining
\begin{equation*}
    \scR_n = \scP_n\backslash \scP_{n-1},
\end{equation*}
and $\scR_0 = \scP_0 = \scP$. 

\begin{remark}
 The name comes from the \say{graphical} idea of a building $\scB$ being impacted by $\mathscr{P}$ creating a crater. The basin is the bottom of the crater formed and the different layers given by our filtration form a topographical map of the crater.
\end{remark}
Using this construction, we define the following concept for a vertex.
\begin{definition}\label{def: height of a vertex}
    Let $(\scB, \scP)$ be an impacted building and let $v$ be any vertex of $\scB$. We define the \defbf{height of v} to be minimal $n$ such that $v\in \scP_n$. Equivalently, the height of $v$ is the unique $n$ such that $v\in\scR_n$. We shall denote the height of $v$ by $h(v)$ or $h_{\scP}(v)$ if we wish to emphasize the basin.
\end{definition}

Using this structure, we now define two generating functions, derived solely from combinatorial data arising in the impacted buildings.

\begin{definition}\label{def: generating functions}
    Let $(\scB, \scP)$ be an impacted building and let $v$ be any vertex of $\scB$. We define the \defbf{layer generating function from v} as
    \begin{equation}\label{eq: layer GF}
        \zeta_v^{\scR}(X) := \displaystyle\sum_{d = 0}^{\infty}\#\left\{x\in \scR_{h(v)} \mid x \mbox{ can be reached by a path of length } d \mbox{ from } v  \right\}X^d.
    \end{equation}
    We also define the \defbf{basin generating function from v} as
    \begin{equation}\label{eq: basin GF}
        \zeta_v^{\scP}(X) := \displaystyle\sum_{d = 0}^{\infty}\#\left\{x\in \scP_{h(v)} \mid x \mbox{ can be reached by a path of length } d \mbox{ from } v  \right\}X^d.
    \end{equation}
    To simplify notation, we shall denote $\zeta_v^{\scR}$ by $\zeta_v$. However, we shall always write $\zeta_v^{\scP}(X)$ to denote the basin generating functions. Also, for the sake of clarity we denote the coefficients of these generating functions by
    \begin{equation*}
        r(d,v)=\#\{x\in \scR_{h(v)}: x \mbox{ can be reached by a path of length } d \mbox{ from } v  \},
    \end{equation*}
    and
    \begin{equation*}
        p(d,v)=\#\{x\in \scP_{h(v)}: x \mbox{ can be reached by a path of length } d \mbox{ from } v  \},
    \end{equation*}
\end{definition}
\begin{remark}
    We emphasize that we are not requiring the path to be nonintersecting nor to be contained entirely $\scR_{h(v)}$ or $\scP_{h(v)}$. 
    
\end{remark}

\begin{example}\label{ex: unram and ram basins}
    The Coxeter Complex of type $\widetilde{A}_1$ is the real line with vertices of two alternating types (which we show by painting them of two colors: red and blue) the type of a vertex is determined by the parity of its distance from $0$. Figure \ref{fig: impacted buildings} shows two \say{impacts} on this building.

    \begin{figure}[ht]
    \centering
        \begin{subfigure}[t]{0.45\textwidth}
            \centering
            \includegraphics[scale = 0.6]{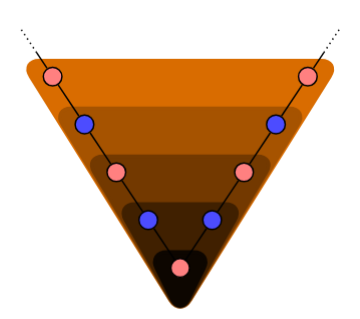}
            \subcaption{\say{Unramified} Impact}
            \label{fig: unramified Impact A2}
        \end{subfigure}
        \begin{subfigure}[t]{0.45\textwidth}
            \centering
            \includegraphics[scale = 0.6]{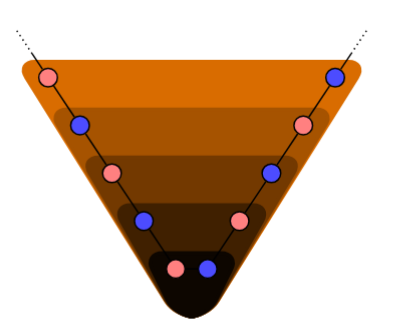}
            \subcaption{\say{Ramified} Impact}
            \label{fig: ramified Impact A2}
        \end{subfigure}
       \caption{Two examples of \protect\say{topographic} maps of impacted buildings.}
            \label{fig: impacted buildings} 
    \end{figure}

In Figure~\ref{fig: unramified Impact A2} we have $\scP$ is a single vertex, while in Figure \ref{fig: ramified Impact A2} we have $\scP$ consists of a chamber. The choice of the name \say{unramified} and \say{ramified} will become clear in Section~\ref{sec: Discussion and proof of the main theorem}. 

Let us compute some layer generating functions in the \say{unramified} case. Fix a half apartment with edge at the point on the basin and call $\cO_n$ its unique vertex in $\scR_n$. In this case, each layer $\scR_n$ consists of exactly two vertices of the same color, except for $\scR_0$ which is only $\cO_0$. In particular, from $\cO_n$ only paths with even length can reach $\scR_n$ again (since type is preserved by even distances). The distance between the two vertices of $\scR_n$ is $2n$, thus any shorter even distance can only reach $\cO_n$ itself. 

The above implies, for $n\ge 1$ that 
\begin{equation*}
    r(d, v)=
    \begin{cases} 
0 & \text{if } d \mbox{ is odd }, \\
1 & \text{if } d \mbox{ is even and } d < 2n,\\
2 & \text{if } d \mbox{ is even and } d \ge 2n,\\
\end{cases}
\end{equation*}
which in turn leads to
\begin{align*}
    \zeta_{\cO_n}(X) &
    = 1 + X^2 + \cdots + X^{2n - 2} + 2X^{2n} + 2X^{2n + 2} + \cdots\\
    &= 1 + X^2 + \cdots + X^{2n - 2} + \dfrac{2X^{2n}}{1 - X^2}\\
    &= \dfrac{1 + X^2 + \cdots + X^{2n - 2} - X^2 - \cdots - X^{2n} + 2X^{2n}}{1 - X^2}\\
    &= \dfrac{1 + X^{2n}}{1 - X^2}
\end{align*}
A similar computation, for $n = 0$, shows
\begin{equation*}
    \zeta_{\cO_0}(X) = \dfrac{1}{1 - X^2}.
\end{equation*}
We leave as an exercise to the interested reader to compute these zeta functions for the \say{ramified} case. In Example \ref{ex: rationality and basin zetas} below, we compute the answer in a slightly shorter way.
\end{example}

\subsection{The Way Out}

The constructions in the previous section only require an impacted building $(\scB, \scP)$. With this information we can construct $\zeta_v(X)$ for all $v$. We now aim to put more structure to the building by constructing a \say{bigger} generating function recursively; recall that $\scB$ is of type $\widetilde{A}_1$.

\begin{definition}\label{def: way out}
   Let $(\scB, \scP)$ be an impacted building and $\cO\in\scP$. We define a \defbf{way out (from $\cO$)} as a half apartment $\scH$ such that
   \begin{enumerate}
       \item its boundary vertex is $\cO$,
       \item for each $n = 0, 1, 2, ...$, $\scH$ there exists a unique vertex $\cO_n$ of height $h(\cO_n) = n$.
   \end{enumerate}
   We denote by $(\scB, \scP, \scH)$ a choice of impacted building with a way out.
\end{definition}
\begin{remark}
    Noitce that a way out does not necessarily exist for every $\cO\in\scP$. In all the cases we will work with it is clear that it does. 
\end{remark}

The purpose of having a way out is to chose, among all the vertices, a particular collection of generating functions from the vertices of $\scH$. That is, we have highlighted the importance of
\begin{equation*}
    \zeta_{\cO_0}, \zeta_{\cO_1}, \zeta_{\cO_2}, \zeta_{\cO_3},... 
\end{equation*}
and of
\begin{equation*}
    \zeta_{\cO_0}^{\scP}, \zeta_{\cO_1}^{\scP}, \zeta_{\cO_2}^{\scP}, \zeta_{\cO_3}^{\scP},... 
\end{equation*}

Before we can establish the fundamental relationship between the above functions defined in this way, we must discuss some simple combinatorial properties of these generating functions.

\subsection{Combinatorics, Rationality and the Recurrence Relation}

In Definition \ref{def: generating functions} the coefficient of $X^d$, in either case, requires a path that is not necessarily a geodesic path. This leads naturally to the related generating functions
\begin{align*}
    \widetilde{\zeta}_v^{\scR}(X) &:= \displaystyle\sum_{d = 0}^{\infty}\left\{x\in \scR_{h(v)} \mid x \mbox{ can be reached by a geodesic path of length } d \mbox{ from } v  \right\}X^d,\\
    \widetilde{\zeta}_v^{\scP}(X) &:= \displaystyle\sum_{d = 0}^{\infty}\left\{x\in \scP_{h(v)} \mid x \mbox{ can be reached by a geodesic path of length } d \mbox{ from } v  \right\}X^d.
\end{align*}
We shall also omit the $\scR$ from the notation and write simply $\wzeta_v$. Similar to Definition~\ref{def: generating functions}, we set $\widetilde{r}(d,v)$ to be the coefficient of $X^d$ in $\widetilde{\zeta}_v^{\scR}(X)$ and $\widetilde{p}(d,v)$ that of $X^d$ in $\widetilde{\zeta}_v^{\scP}(X)$.

\begin{proposition}\label{prop: geodesic non geodesic}
   Let $(\scB,\scP)$ be an impacted building, and $v$ a vertex of $\scB$, then
    \begin{equation*}
        \zeta_v(X) = \dfrac{\wzeta_v(X)}{1 - X^2} \quad\mbox{and}\quad  
        \zeta_v^{\scP}(X) = \dfrac{\wzeta^{\scP}_v(X)}{1 - X^2}.
    \end{equation*}
\end{proposition}
\begin{proof}
    Both equalities are proven using the same method; we will consider only the first. Set $n = h(v)$; then when counting coefficients for $\widetilde{\zeta}_v(X)$, for each $x \in \scR_{n}$, it will only be counted once: exactly when $d = d(x,v)$. Hence, we have that 
    $$\widetilde{\zeta}_v(X) = \sum_{x\in \scR_n} X^{d(x, v)}.$$
    By our previous comment, we notice that when $d \in d(x,v) + 2\Z_{\geq 0}$, there is a path of length $d$ connecting $x$ and $v$. This shows that
    \begin{equation*}
        \zeta_v(X) = \displaystyle\sum_{x\in\scR_n}\sum_{\ell\in 2\Z_{\geq 0}}X^{d(x,v) + \ell} = \dfrac{1}{1 - X^2}\displaystyle\sum_{x\in\scR_n}X^{d(x, v)} = \dfrac{\wzeta_v(X)}{1 - X^2},
    \end{equation*} concluding the proof.
\end{proof}

\begin{corollary}\label{cor: rationality finite basin}
    Let $(\scB, \scP)$ be an impacted building with $\scP$ finite, then $\zeta_v(X)$ and $\zeta_{v}^{\scP}(X)$ are rational functions.
\end{corollary}
\begin{proof}
    If $\scP$ is finite, then every $\scP_n$ is finite. In particular, each of them has a finite diameter. This implies that $\wzeta_v$ and $\wzeta_v^{\scP}$ are polynomials because the coefficient of $X^d$ vanishes for $d$ greater than the diameter. This concludes the proof.
\end{proof}
\begin{example}\label{ex: rationality and basin zetas}
   Let us return to the constructions of Example \ref{ex: unram and ram basins}, to highlight how Corollary~\ref{cor: rationality finite basin} applies to save computational effort. In the \say{unramified} impact, we have seen that for $n\ge 1$ there are two vertices of the same color in $\scP_n$. One is at distance $0$ from $\cO_n$. The other is at distance $2n$. Thus
   \begin{equation*}
       \wzeta_{\cO_n} = 1 + X^{2n}.
   \end{equation*}
   Corollary \ref{cor: rationality finite basin} implies
   \begin{equation*}
       \zeta_{\cO_n} = \dfrac{1 + X^{2n}}{1 - X^2},
   \end{equation*}
   agreeing with our previous work. We also know that there is exactly one vertex at each of the distances $0, 1, 2,..., 2n$ in $\scP_n$. Thus
   \begin{equation*}
       \wzeta_{\cO_n}^{\scP}(X) = 1 + X + X^2 + \cdots + X^{2n} = \dfrac{1 - X^{2n + 1}}{1 - X}.
   \end{equation*}
   As a consequence, we get
   \begin{equation*}
       \zeta_{\cO_n}^{\scP}(X) = \dfrac{1 + X + X^2 + \cdots + X^{2n}}{1 - X^2} = \dfrac{1 - X^{2n + 1}}{(1 - X)(1 - X^2)}.
   \end{equation*}
   In the \say{ramified} impact, we have that from $\cO_n$ there is exactly a vertex at distance $0, 1, ..., 2n + 1$ in $\scP_n$. Thus
   \begin{equation*}
       \zeta_{\cO_n}^{\scP}(X) = \dfrac{1 + X + \cdots +X^{2n + 1}}{1 - X^2} = \dfrac{(1 + X)(1 + X^2 + \cdots +X^{2n})}{1 - X^2} = \dfrac{1 + X^2 + \cdots +X^{2n}}{1 - X}.
   \end{equation*}
\end{example}

The functions $\zeta_{\cO_n}$ also admit a recursive formula.

\begin{theorem}\label{thm: recurrence relation genfun}
    Let $(\scB, \scP, \scH)$ be an impacted building with way out, then for $n\ge 0$ we have
    \begin{equation*}
        \zeta_{\cO_{n+1}}^{\scP}(X) = \zeta_{\cO_{n+1}}(X) + X\zeta_{\cO_{n}}^{\scP}(X).
    \end{equation*}
\end{theorem}
\begin{proof}
    Using Proposition \ref{prop: geodesic non geodesic}, we see it is enough to prove
   \begin{equation*}
       \wzeta_{\cO_{n+1}}^{\scP}(X) = \wzeta_{\cO_{n+1}}(X) + X\wzeta_{\cO_{n}}^{\scP}(X).
   \end{equation*}
   Let $d\geq 1$ be given; comparing coefficients of $X^d$ and unraveling the definition of $\widetilde{\zeta}_{v}(X)$ yields that an equivalent formulation of this is 
   $$\widetilde{p}(\cO_{n+1}, d) = \widetilde{r}(\cO_{n+1}, d) + \widetilde{p}(\cO_{n}, d-1).$$
   We now notice that $\scP_{n+1} = \scR_{n+1} \sqcup \scP_{n}$. Because we are counting geodesics and $\scB$ is a tree, the only way from $\cO_{n+1}$ to enter $\scP_n$ is through $\cO_n$. Moreover, such paths must immediately go from $\cO_{n+1}$ to $\cO_n$ to prevent self-intersection. This defines a bijection between paths contributing to $\widetilde{p}(\cO_{n+1}, d)$ not contributing to $\widetilde{r}(\cO_{n+1}, d)$ and paths contributing to $\widetilde{p}(\cO_n, d-1)$, yielding our result.
\end{proof}
\section{Explicit computations of generating functions}\label{sec: Explicit computations of generating functions}

The purpose of this section is to study three families of impacted buildings that arise naturally in the study of the zeta functions of orders discussed in the introduction. We call them the \say{unramified,} \say{ramified,} and \say{split} cases due to their connection with the zeta functions of orders. Specifically, the zeta functions of each of these cases corresponds to an impacted building of the corresponding family. In this section, we continue to work exclusively with the buildings without the appearance of arithmetic. 

\begin{definition}
    Let $m\ge 2$ be a positive integer. We denote by $\scT_m$ the homogeneous tree of degree $m + 1$.
\end{definition}

For each of the above trees we now define an impacted building with way out structure. We describe the basins first.

\begin{definition}\label{def: impacted buildings cases defs}
    Let $m\ge 2$ be a positive integer. We define the following basins:
    \begin{enumerate}
        \item The unramified basin $\scP_U$ consists of a single vertex.
        \item The ramified basin $\scP_R$ consists of a single edge.
        \item The split basin $\scP_S$ consists of a single apartment.
    \end{enumerate}
    We denote the impacted building $(\scT_m, \scP_U)$, $(\scT_m, \scP_R)$ and $(\scT_m, \scP_S)$ by $\scB_U(m)$, $\scB_R(m)$ and $\scB_S(m)$, respectively. We call them the unramified, ramified and split impacted buildings of degree $m + 1$. We will sometimes omit the word \say{impacted} as there will be no possibility of confusion. When we do not need to specify a particular case we will drop the subindex.
\end{definition}

In each of the above cases, we can fix a vertex $\cO\in\scP$ and follow a half apartment from $\cO$ away from $\scP$. Notice, we cannot do this if $m = 1$ in the split case, which explains our choice of $m\ge 2$. For the ramified and unramified case, we can extend our definitions. We have previously seen these constructions in Examples~\ref{ex: unram and ram basins} and~\ref{ex: rationality and basin zetas}. In what follows, we will fix a way out once and for all, as it will become clear that the results that we obtain do not depend on this choice. The corresponding impacted buildings may be seen in figure~\ref{basinsInEachCase}.

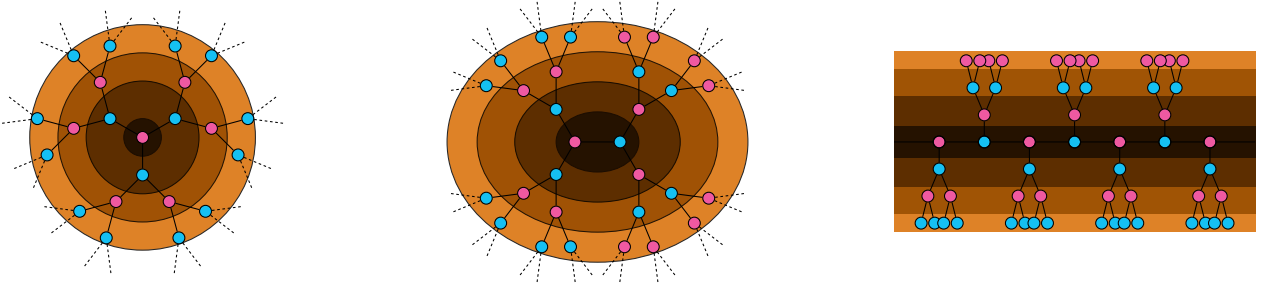
\begin{figure}[h]
\centering
\begin{minipage}{0.33\textwidth}
\begin{tikzpicture}[scale = 0.5,
    grow cyclic,
    level distance=1cm,
    level 1/.style={sibling angle=120},
    level 2/.style={sibling angle=90},
    level 3/.style={sibling angle=60},
    level 4/.style={sibling angle=45, dash pattern=on 1pt off 1pt},
    nodes={circle, draw, inner sep=0pt, minimum size=4.5pt} 
  ]

\draw[fill={rgb, 255:red, 217; green,108; blue,0}, opacity=0.85] (0,0) circle(3cm);

\draw[fill={rgb, 255:red, 149; green,74; blue,0}, opacity=0.85] (0,0) circle(2.25cm);

\draw[fill={rgb, 255:red, 81; green,40; blue,0}, opacity=0.85] (0,0) circle(1.5cm);

\draw[fill={rgb, 255:red, 13; green,6; blue,0}, opacity=0.7 ] (0,0) circle(0.5cm);

\path[rotate=30]
  node[fill=magenta!80] {}   child foreach \cntI in {1,...,3} {    node[fill=cyan!70] {}    child foreach \cntII in {1,...,2} {       node[fill=magenta!80] {}      child foreach \cntIII in {1,...,2} {        node[fill=cyan!70] {}        child foreach \cntIV in {1,...,2} {}}}};

\end{tikzpicture}
\end{minipage}
\begin{minipage}{0.33\textwidth}
\begin{tikzpicture}[scale=0.5,
    grow cyclic,
    level distance=1cm,
    level 1/.style={sibling angle=120},
    level 2/.style={sibling angle=60},
    level 3/.style={sibling angle=45},
    level 4/.style={sibling angle=30, dash pattern=on 1pt off 1pt},
    nodes={circle, draw, inner sep=0pt, minimum size=4.5pt} 
  ]
  \draw[fill={rgb, 255:red, 217; green,108; blue,0}, opacity=0.85] (0, 0) ellipse (4cm and 3.2cm);
  
  \draw[fill={rgb, 255:red, 149; green,74; blue,0}, opacity=0.85] (0, 0) ellipse (3.2cm and 2.4cm);

  \draw[fill={rgb, 255:red, 81; green,40; blue,0}, opacity=0.85] (0, 0) ellipse (2.2cm and 1.6cm);

  \draw[fill={rgb, 255:red, 13; green,6; blue,0}, opacity=0.7] (0, 0) ellipse (1.1cm and 0.8cm);

\draw[black] (-0.6, 0) -- (0.6, 0);
  \path[shift = {(0.6,0)}, rotate=0]
node[fill=cyan!70] {}   child foreach \cntI in {1,...,2} {    node[fill=magenta!80] {}    child foreach \cntII in {1,...,2} {       node[fill=cyan!70] {}      child foreach \cntIII in {1,...,2} {        node[fill=magenta!80] {}        child foreach \cntIV in {1,...,2} {}}}};

\path[shift = {(-0.6,0)}, rotate=-180]
node[fill=magenta!80] {}   child foreach \cntI in {1,...,2} {    node[fill=cyan!70] {}    child foreach \cntII in {1,...,2} {       node[fill=magenta!80] {}      child foreach \cntIII in {1,...,2} {        node[fill=cyan!70] {}        child foreach \cntIV in {1,...,2} {}}}};

\end{tikzpicture}
\end{minipage}
\begin{minipage}{0.3\textwidth}
\begin{tikzpicture}[scale = 0.3,
  every node/.style={circle, draw, inner sep=0pt, minimum size=4.5pt},
  level distance=12mm,
  level 1/.style={sibling distance=14mm},
  level 2/.style={sibling distance=10mm},
  level 3/.style={sibling distance=6mm},
  level 4/.style={
    level distance=6mm,
    edge from parent/.style={dash pattern=on 1pt off 1pt}
  }
]

\fill[fill={rgb, 255:red, 217; green,108; blue,0}, opacity=0.85] (-8,-40mm) rectangle (8,40mm);
\fill[fill={rgb, 255:red, 149; green,76; blue,0}, opacity=0.85] (-8,-32mm) rectangle (8,32mm);
\fill[fill={rgb, 255:red, 81; green,40; blue,0}, opacity=0.85] (-8,-20mm) rectangle (8,20mm);
\fill[fill={rgb, 255:red, 13; green,6; blue,0}, opacity=0.7] (-8,-7mm) rectangle (8,7mm);

\draw (-8,0) -- (8,0);

\foreach \i in {-3,...,3} {
  \ifodd\numexpr\i+3\relax
    \node[fill=cyan!70] (root\i) at (\i*2,0) {};
  \else
    \node[fill=magenta!80] (root\i) at (\i*2,0) {};
  \fi
}

\foreach \i/\dir in {-3/down, -2/up, -1/down, 0/up, 1/down, 2/up, 3/down} {

  \ifodd\i
    \tikzset{
      level 1/.append style={fill=cyan!70},
      level 2/.append style={fill=magenta!80},
      level 3/.append style={fill=cyan!70}
    }
  \else
    \tikzset{
      level 1/.append style={fill=magenta!80},
      level 2/.append style={fill=cyan!70},
      level 3/.append style={fill=magenta!80}
    }
  \fi

  \begin{scope}[shift={(root\i)}]
    \node {}
      child[grow=\dir] {
        node[level 1] {}
        child {
          node[level 2] {}
          child {
            node[level 3] {}
            child[level 4] { node[draw=none, minimum size=0pt] {} }
            child[level 4] { node[draw=none, minimum size=0pt] {} }
          }
          child {
            node[level 3] {}
            child[level 4] { node[draw=none, minimum size=0pt] {} }
            child[level 4] { node[draw=none, minimum size=0pt] {} }
          }
        }
        child {
          node[level 2] {}
          child {
            node[level 3] {}
            child[level 4] { node[draw=none, minimum size=0pt] {} }
            child[level 4] { node[draw=none, minimum size=0pt] {} }
          }
          child {
            node[level 3] {}
            child[level 4] { node[draw=none, minimum size=0pt] {} }
            child[level 4] { node[draw=none, minimum size=0pt] {} }
          }
        }
      };
  \end{scope}
}
\end{tikzpicture}
\end{minipage}
\caption{The impacted buildings in the unramified, ramified, and split case from left to right.}
\label{basinsInEachCase}
\end{figure}



\subsection{The unramified case}

In this case, we think of $\scB_U(m)$ as a tree with a single root $\cO_0$. One immediately sees that $\scP_n$ consists of all vertices at distance at most $n$ from $\cO_0$. In particular, the vertices at distance exactly $n$ are those vertices in $\scR_n$. All of them will be of the same type and the maximum distance between any such vertices is $2n$.

\begin{lemma}\label{lem: unram countOfVerts}
    Let $v\neq \cO_0$ be any vertex of $\scB_U(m)$, then 
    \begin{equation*}
        \mid\!\scR_{h(v)}\!\mid = (m+1)m^{h(v)-1}
    \end{equation*}
\end{lemma}
\begin{proof}
    In this case, $\scR_{h(v)}$ consists of the vertices at distance $h(v)$ from $\cO_0$. By this, we count $m + 1$ branches out of $\cO_0$. Moreover, for the next layer, each vertex adds $m$ new vertices per branch. Thus, per possible branch, we get  $m^{h(v)-1}$ vertices in the $h(v)$-layer; the result follows.  
\end{proof}

\begin{proposition}\label{prop: UnramifiedCoeffs}
    Let $v\neq \cO_0$ be a vertex in $\scB_U(m)$. The values $r(d,v)$ are given by 
    \begin{equation*}
        r(d,v) = \begin{cases}
        0 & \text{if $d$ is odd,}\\
        m^{d/2} & \text{if }d < 2h(v),\, \text{and $d$ is even,}\\
        (m+1)m^{h(v)-1} & \text{if } d\geq 2h(v),\, \text{and $d$ is even.}
    \end{cases}
    \end{equation*}
\end{proposition}
\begin{proof}
    Because all the vertices in $\scR_{h(v)}$ have the same color, they must be at even distance one from each other. This proves the statement for $d$ odd. On the other hand, when $d\geq 2h(v)$ is even, every vertex in $\scR_{h(v)}$ can be reached. Thus Lemma \ref{lem: unram countOfVerts} implies the result in this case.

    We proceed now to when $d<2h(v)$ and even; write $d=2k$. Then, there exists a unique branch of $\scB_{U}$ whose intersection with $\scR_{h(v)}$ is at distance exactly $2h(v)$ from $v.$ By the above logic, we may reach every vertex of distance $v-k$ from $k$, except those on the one branch defined by this unique point. This proves the theorem, applying Lemma~\ref{lem: unram countOfVerts} once more.
\end{proof}

We may now compute the basin generating functions associated to our way out. 

\begin{proposition}\label{prop: unram gen fun}
    In the unramified case, we have
    \begin{equation*}
        \zeta_{\cO_n}(X) = \dfrac{1 + (m-1)X^2 + m(m-1)X^4 + \cdots + m^{n-2}(m-1)X^{2(n-1)}+m^{n}X^{2n}}{1-X^2}.
    \end{equation*}
    and
    \begin{equation*}
        \zeta_{\cO_n}^{\scP}(X) = \dfrac{1+ X + m(X^{2}+ X^{3}) + \cdots + m^{n}(X^{2n-2}+ X^{2n-1}) + m^nX^{2n}}{1 - X^2}.
    \end{equation*}
\end{proposition}
\begin{proof}
 By Proposition \ref{prop: UnramifiedCoeffs},  
\begin{align*}
    \zeta_{\cO_{h(v)}} &= 1 + mX^2 + m^2X^4 + \cdots + m^{n-1}X^{2n-2} + (m+1)m^{n-1}X^{2n}(1 + X^2 + X^4 + \cdots )\\
    &= 1 + mX^2 + m^2X^4 + \cdots + m^{n-1}X^{2n-2} + \frac{(m+1)m^{n-1}X^{2n}}{1-X^2}\\
    &=\dfrac{1 + (m-1)X^2 + m(m-1)X^4 + \cdots + m^{n-2}(m-1)X^{2(n-1)}+m^{n}X^{2n}}{1-X^2}.
\end{align*}  
Using Theorem \ref{thm: recurrence relation genfun} recursively gives
\begin{equation*}
    \zeta_{\cO_n}^{\scP} = \sum_{i=0}^{n} X^i\zeta_{\cO_{n-i}}
\end{equation*}
Substituting the values for each of $\zeta_{\cO_{n-i}}$ in the above equation and chasing powers of $X$ (which amounts to a tedious but straightforward calculation) yields
\begin{equation*}
        \zeta_{\cO_n}^{\scP}(X) = \dfrac{1+ X + m(X^{2}+ X^{3}) + \cdots + m^{n}(X^{2n-2}+ X^{2n-1}) + m^nX^{2n}}{1 - X^2}.
\end{equation*}
This concludes the proof.
\end{proof}
\begin{remark}
    By Proposition \ref{prop: geodesic non geodesic}, the above implies that in the unramified case
    \begin{equation*}
        \wzeta_{\cO_n}^{\scP}(X) = 1+ X + m(X^{2}+ X^{3}) + \cdots + m^{n-1}(X^{2n-2}+ X^{2n-1}) + m^nX^{2n}.
    \end{equation*}
    One can thus interpret these coefficients as the number of vertices at the corresponding exact distances from $\cO_n$. We leave to the reader to carry out this interpretation.
\end{remark}

\subsection{The ramified case}
Continuing with the strategy outlined at the start of this section, we note that $\mathscr{B}_R(m)$ can be thought of as the tree $\mathcal{T}_m$ with basis the chamber defined by $\cO_0, \cO_1$. In particular, we notice that $\scP_n$ consists of $m$ branches from $\cO_0$ and $m$ branches from $\scP_n$; if we note that no branches of an $\cO_i$ contain an $\cO_j$, this gives a well defined description. 

\begin{lemma}\label{L: countInRamifiedCase}
    Let $v \not\in \mathscr{P}_R$ be any vertex of $\mathscr{B}_R(m)$, then
    $$|\mathscr{R}_{h(v)}| = 2m^{h(v)}.$$
\end{lemma}
\begin{proof}
    The set of vertices in $\scR_{h(v)}$ is given by $m$ branches of the ball $B(\cO_0;h(v))$ and $m$ branch of the ball $B(\cO_1;h(v))$. These sets are clearly disjoint as the distances of their elements from $0$ have distinct parity. Hence, the result follows.
\end{proof}

\begin{proposition}\label{ramifiedCoeffs}
    Let $v\in \scB_R(m)$ and $v\not\in \scP$ be given. The values $r(d,v)$ are given by 
    $$r(d,v)= \begin{cases}
        0 & d=2k+1,\, k < h(v)\\
        m^{k} & d = 2k,\, k < h(v)\\
        m^{h(v)} & d \geq 2h(v)\\
    \end{cases}.$$
\end{proposition}
\begin{proof}
    To begin, we notice that vertices on opposite sides on the basin $\scP$ have opposite type; hence if $d$ is odd, paths must go through the basin, while if $d$ they must not. Immediately this gives that if $d$ is odd and $d<2k$, then there are no such paths. The case of $d=2k,$ $k<h(v)$ is the same as the proof of Theorem~\ref{lem: unram countOfVerts}.

    Now when $d\geq 2h(v),$ the vertices are exactly those contained in the intersection $\scR_n$ and of distance $h(v)$ from $\cO_{j(d)}$, where $j(d)=2(\frac{d}{2} - \lfloor\frac{d}{2}\rfloor)$,
    as the diameter of $\scP_n$ is exactly $h(v).$ Thus the proof of Lemma~\ref{L: countInRamifiedCase} applies.
\end{proof}

We may now compute the basin generating functions associated to our way out.

\begin{proposition}
    In the ramified case, we have 
    \begin{align*}\zeta_{\cO_n}(X) &= \frac{1-X + mX^2- mX^3 + \cdots +m^{n-1} X^{2n-2} - m^{n-1} X^{2n-1} + m^n X^{2n}}{1-X}\\
    \zeta_{\cO_{n}}^P &= \frac{1+mX^2 + m^2X^4 + \cdots + m^nX^{2n}}{1-X}.
    \end{align*}
\end{proposition}
\begin{proof}
    By Lemma~\ref{ramifiedCoeffs}, we find that 
    \begin{align*}
    \zeta_{\cO_{n}}(X) &= 1 + mX^2 + m^2X^4 + \cdots + m^{n-1}X^{2n-2} + m^{n}(X^{2n} + X^{2n+1} + \cdots)\\
    &= \frac{1}{1-X}((1-X)(1+ mX^2 + m^2X^4 + \cdots + m^{n-1}X^{2n-2}) + m^nX^{2n})\\
    &= \frac{1}{1-X}(1-X + mX^2- mX^3 + \cdots + m^{n-1} X^{2n-2} - m^{n-1} X^{2n-1} + m^n X^{2n}).
\end{align*}
Moreover, using Theorem~\ref{thm: recurrence relation genfun} yields that $$\zeta_{\cO_n}^P(X) = \sum_{i=0}^n X^i\zeta_{\cO_{n-i}};$$
by isolating coefficients in this sum, the result for $\zeta_{\cO_n}^P(X)$ follows.
\end{proof}

\subsection{The split case}
Before beginning with the split case, we emphasize that here the basin $\scP$ is itself infinite, being an apartment.  Moreover, the orders $\cO_n$ are associated to a half apartment out of $\cO_0$, a point in the apartment.  

\begin{lemma}\label{L: coeffs split}
   Let $v\in \scR_n$ be given. The numbers $r(d,v)$ are given by
    $$r(d,v) = \begin{cases}
        m^{k-1} & d = 2k < 2n\\
        0 & d = 2k+1 < 2n\\
        \ell(m-1)m^{n-1} & d = 2n + \ell
    \end{cases}$$
\end{lemma}
\begin{proof}
    The distance from $v$ to the basin is $h(v)$; let us write $n=h(v)$. Hence to enter the basin and leave it again, returning at the same height a distance of at least $2n$ must be traveled. In particular, we find that if $d < 2n$ then the only vertices of $\scR_n$ that may be reached are the ones on the same branch as $v.$ 

    Now when $d \geq 2n,$ one can move $\ell = d - 2n$ units along the fundamental basin. This leads to $\ell$ vertices to leave the basin on ($\ell/2$ in each direction). The result follows upon realizing there are $m-1 = m + 1       - 2$ branches out from each vertex.
\end{proof}

Once again, this data is sufficient to compute the zeta functions using our way out. 
\begin{proposition}
    In the split case, we have \begin{align*}
        \zeta_{\cO_n}(X) &= \frac{1 - 2X + (m-1)X^2 - 2X^3 + (m-1)X^4 - \cdots + m^{n-2}(m-1)X^{2(n-1)} - 2X^{2n-1}+m^nX^{2n}}{(1-X)^2}\\
        \zeta_{\cO_n}^P(X) &= \frac{1 - X + m(X^2 - X^3) + \cdots + m^n(X^n - X^{n+1})}{(1-X)^2}.
    \end{align*}
\end{proposition}
\begin{proof}
    Using Lemma~\ref{L: coeffs split}, one computes \begin{align*}\zeta_{\cO_n} &= 1 + mX^2 + \cdots + m^{n-1}X^{2n-2} + (m-1)m^{n-1}X^{2n} + 2(m-1)m^{n-1}X^{2n+1} + \cdots \\
    &= 1 + mX^2 + \cdots + m^{n-1}X^{2n-2} + \frac{(m-1)m^{n-1}X^{2n}}{(1-X)^2}\\
    &= \frac{1 - 2X + (m-1)X^2 - 2X^3 + (m-1)X^4 - \cdots + m^{n-2}(m-1)X^{2(n-1)} - 2X^{2n-1}+m^nX^{2n}}{(1-X)^2}.\end{align*}
    Now, using the recursive formula obtained from Theorem~\ref{thm: recurrence relation genfun} and isolating for coefficients in the numerator yields the result about $\zeta_{\cO_n}^P(X)$, concluding our proof.
\end{proof}

\section{Order polynomials for $\GL(2)$  and their computations via ideal types}\label{sec: Order polynomials for GL(2)  and their computations via ideal types}

In this section, we present the constructions used in~\cite{MEELZETA} to compute the polynomials mentioned in Section~\ref{sec: Introduction}, via ideals and their type. Throughout this section, our treatment will be brief, presenting only the main statements; the interested reader can refer to~\cite{MEELZETA} for these proofs. The results of this section will be used in Section~\ref{sec: Unit distribution and natural appearance of Impacted Buildings} and to prove our main theorem in~\ref{sec: Discussion and proof of the main theorem}.

\subsection{The setup and the Zeta Function of Orders}

Let $K$ be a nonarchimedean local field of characteristic zero and $\cO_K$ its ring of integers. Let $p$ be a uniformizer of $K$ and $q$ the cardinality of the residue field of $K$. Let $L$ be a reduced $K$-algebra of dimension $2$ over $K$ and $\cO_L$ the integral closure of $\cO_K$ in $L$. Under these hypotheses, there are three possibilities for $L$: an unramified quadratic extension of $K$, a ramified quadratic extension of $K$, or $K\times K$. We refer to these cases as the \textit{unramified}, \textit{ramified} and \textit{split} case, respectively. When $L$ is a field, we shall denote by $\pi$ a fixed uniformizer of $L$ and by $e$ and $f$ the ramification and inertia degrees of $L/K$ respectively.

The general theory of discrete valuation rings guarantees the existence of $\Delta\in\cO_K$ such that
\begin{equation*}
    \cO_L = \cO_K[\Delta].
\end{equation*}
This implies that $L = K(\Delta)$. Based on this we can define
\begin{equation*}
    \cO_n := \cO_K[p^n\Delta]\subset\cO_L,
\end{equation*}
for $n = 0, 1, 2, \cdots$. Each of these sets is an order of $\cO_L$ and we refer to them collectively as the \textit{main sequence of orders}. It is an easy exercise to prove that $\cO_n$ is independent of $\Delta$. In this section we shall have no need to specify a particular $\Delta$. However, in Section \ref{sec: Unit distribution and natural appearance of Impacted Buildings} we will pick for each case convenient choices.

In \cite{ZYun}, Yun defines for (general) orders a zeta function and proves that it is a rational function of $q^{-s}$. For the orders $\cO_n$ of our main sequence, it can be verified that these zeta functions reduce to
\begin{equation*}
    \zeta_{\cO_n}(s) = \displaystyle\sum_{I\subset\cO_n}\dfrac{1}{[\cO_n: I]^s},
\end{equation*}
where the sum runs over ideals $I\subseteq\cO_n$ of finite index. One can verify that the ideals of finite index are exactly those that have rank $2$ as $\cO_K$-modules. We shall always assume that the ideals we are working with satisfy this latter condition.

The aforementioned zeta functions are thus the classical ideal counting functions. In particular, when $L$ is a field extension and $n = 0$ it coincides with its local zeta function. 

The rationality of these zeta functions implies there are two polynomials, $P(X)$ and $V(X)$ such that
\begin{equation*}
    \zeta_{\cO_n}(s) = \dfrac{P(q^{-s})}{V(q^{-s})}.
\end{equation*}
The denominator polynomial depends uniquely on the extension $L/K$ and is easy to compute. On the other hand, the numerator polynomial $P$ depends on the particular order $\cO_n$. Their values are given by the following result.


\begin{theorem}\label{thm: polynomialvalues}
For each $n\ge 0$ define the following polynomials:
\begin{align*}
        R_n(X) &= 1 + qX^2 + q^2X^4 + ... + q^nX^{2n},
\end{align*}
and for $n\ge 1$ define
\begin{align*}
    U_n(X) &= (1 + X)R_{n - 1}(X) + q^nX^{2n},\\
    S_n(X) &= (1 - X)R_{n - 1}(X) + q^nX^{2n}.
\end{align*}
Finally, also put $U_0(X) = S_0(X) = 1$. Explicitly, these polynomials are
\begin{align*}
    R_n(X) &= 1 + qX^2 + q^2X^4 + ... + q^nX^{2n},\\
    U_n(X) &= 1 + X + qX^2 + qX^3 + ... + q^{n - 1}X^{2n - 2} + q^{n - 1}X^{2n - 1} + q^nX^{2n},\\
    S_n(X) &= 1 - X + qX^2 - qX^3 + ... + q^{n - 1}X^{2n - 2} - q^{n - 1}X^{2n - 1} + q^nX^{2n}.
\end{align*}Then we have 
\begin{align*}
    (1 - q^{-s})\zeta_{\cO_n}(s) &= R_n(q^{-s}),\\
    (1 - q^{-2s})\zeta_{\cO_n}(s) &= U_n(q^{-s}),\\
    (1 - q^{-s})^2\zeta_{\cO_n}(s) &= S_n(q^{-s}).
\end{align*}
That is, $R_n, U_n$ and $S_n$ are the numerators of the corresponding zeta functions in the ramified, unramified and split case, respectively.
\end{theorem}
\begin{proof}
    This is Theorem 5.4 of \cite{MEELZETA}.
\end{proof}

\subsection{The Traveling map and the Principal Zeta Function}

In \cite{MEELZETA} the strategy to prove Theorem \ref{thm: polynomialvalues} requires two steps. The first one establishes a recurrence relation between different zeta functions. This step has as consequence the reduction of the problem to that of computing the contribution from the principal ideals. The second one consists on computing the latter by organizing the principal ideals by their type.
We begin with discussing the recurrence relation.

\begin{definition}
For each integer $n\ge 0$, let $\cI_n$ be the set of $\cO_K$-rank 2 ideals of $\cO_n$. Define the \textbf{traveling map}
\begin{equation*}
    T_n: \cI_n \longrightarrow \cI_{n + 1},\quad
    T_n(J) = pJ.
\end{equation*}
\end{definition}

This definition is motivated by the following theorem.
\begin{theorem}\label{thm: travelingmap}
The traveling map $T_n$ is a bijection onto its image, which consists exactly of the non-principal ideals in $\cI_{n+1}$.
\end{theorem}
\begin{proof}
    This is Theorem 3.22 in \cite{MEELZETA}.
\end{proof}
This motivates us to isolate the contribution to $\zeta_{\cO_n}(s)$ from the principal ideals. 
\begin{definition}\label{def: principal ideal zeta function}
    For each $n\ge 0$, we define the \defbf{principal ideal zeta function} as
    \begin{equation*}
    \zeta_{\cO_n}^P(s) := \sum_{\substack{I \subseteq \cO_n \\ \text{principal}}} \frac{1}{[\cO_n : I]^s}
\end{equation*}
\end{definition}
A consequence of Theorem \ref{thm: travelingmap}, together with some index manipulations, is the recurrence relation we mentioned above. 
\begin{theorem}\label{thm: recurrencestatement}
For $n\ge 1$, the zeta functions of the orders $\cO_n$ satisfy the recurrence relation
\begin{equation}\label{eq: recurrence relation zetas}
    \zeta_{\cO_n}(s) = \zeta_{\cO_n}^P(s) + q^{-s}\zeta_{\cO_{n-1}}(s).
\end{equation}
\end{theorem}
\begin{proof}
    This is Theorem 4.2 in \cite{MEELZETA}.
\end{proof}

\subsection{Types of Principal Ideals}

If one can compute $\zeta_{\cO_n}^P(s)$ then, by the repeated use of the recurrence relation \eqref{eq: recurrence relation zetas}, one can compute $\zeta_{\cO_n}(s)$. In order to do this one introduces the type of an element in $\cO_L$. 

\begin{definition}\label{def: type of an element}
    Let $x\in \cO_0 = \cO_L$ be any element. We define its \textbf{type} by
\begin{equation*}
    \varepsilon(x) = 
    \left\{
	\begin{array}{ll}
		\valpi(x)  & \mbox{in the nonsplit case }  \\
		(\valp(x_1), \valp(x_2)) & \mbox{in the split case }  
	\end{array}
\right.
\end{equation*}
where in the split case we have $x = (x_1, x_2)$. We define a \textbf{possible type} as an element of the image of $\varepsilon$. We also define for every $x\in \cO_L$
\begin{equation*}
    \eta(x) = \min\{\valp(a) \mid a \mbox{ is a coordinate of } \varepsilon(x)\}.
\end{equation*}

Furthermore, let $I\subseteq\cO_n$ be a principal ideal and $\alpha$ one of its generators. We define the \textbf{type of $I$} by
\begin{equation*}
    \varepsilon(I):= \varepsilon(\alpha),
\end{equation*}
and $\eta(I) := \eta(\alpha)$. It is straightforward to verify these definitions are independent of the generator $\alpha$ of $I$.
\end{definition}

The types thus correspond to the ideals that can contribute to the sum $\zeta_{\cO_n}^P$. In practice, the contribution of ideals is determined by their type. Moreover, the contribution of types is not uniform; rather, they are classified in two main families, depending only on the \say{magnitude} of a type.

\begin{definition}\label{def: high and low types}
Let $\omega$ be a possible type. In the nonsplit case, we say $\omega$ is a \textbf{low type} if $\omega < ne$ where $e$ is the ramification degree of $L/K$. In the split case, we say $\omega = (\omega_1, \omega_2)$ is a \textbf{low type} if $\omega_1 < n$ or $\omega_2 < n$.
Otherwise, we say $\omega$ is a \textbf{high type}. 
\end{definition}
Adapting the above definition to ideals leads to the following definition.
\begin{definition}\label{def: high and low ideals}
Let $\omega$ be a possible type of a principal ideal of $\cO_n$. We say $I$ is a \textbf{low ideal} if $\varepsilon(I)$ is a low type. That is, if
\begin{equation*}
    \eta(I) < \left\{
	\begin{array}{ll}
		ne & \mbox{in the nonsplit case, }  \\
		n & \mbox{in the split case. }  
	\end{array}
\right.
\end{equation*}
Otherwise, we say $I$ is a \textbf{high ideal}. We will refer to $ne$ and $n$, respectively in the nonsplit and split cases, as the \textbf{ideal threshold} and denote if uniformly by $t_n$.
\end{definition}

\subsection{The computation of the Zeta Function of Orders}

With the concept of low and high ideals at hand, we can compute the principal idea zeta function. First, we notice that the index $[\cO_n: I]$ is uniquely determined by $\epsilon(I)$.

\begin{definition}\label{def: type contribution}
Let $\omega$ be a possible type of a principal ideal $I\subseteq \cO_n$. We define its \textbf{ type contribution} as
\begin{equation*}
    c(\omega):= \left\{
	\begin{array}{ll}
		f\omega & \mbox{in the nonsplit case, }  \\
		\omega_1 + \omega_2 & \mbox{in the split case. }  
	\end{array}
\right.
\end{equation*}
In here we are denoting $\omega = (\omega_1, \omega_2)$ in the split case, where the possible types are pairs of nonnegative integers. 
\end{definition}
In terms of the type contribution one easily proves
\begin{equation}\label{eq: index value contribution}
    [\cO_n:I] = q^{-c(\varepsilon(I))}.
\end{equation}

Let $\omega$ be a possible type and define 
\begin{equation*}
   X_{\omega} :=  \{I\subseteq\cO_n, \mbox{ principal ideal } \mid\; \varepsilon(I) = \omega\}. 
\end{equation*}
Equation~\eqref{eq: index value contribution} implies that every ideal in $X_{\omega}$ contributes the same to the zeta function. Thus, grouping ideals by type, we deduce
\begin{equation*}
    \zeta_{\cO_n}^P(s) = \displaystyle\sum_{\omega}\dfrac{| X_{\omega}|}{q^{c(\omega)}}.
\end{equation*}

With this, we have reduced the computation of the principal zeta function to a two-step-process: \begin{itemize}
    \item deduce which possible types appear as types of an ideal, and
    \item for each of these types, compute $|X_{\omega}|$.
\end{itemize}
This process is summarized in the next theorem.

\begin{theorem}\label{thm:Type situation} 
The criteria to determine which possible types are types of principal ideals and their contribution are given by the following facts.
    \begin{description}
        \item[High ideals:] Every possible high type $\omega$ is the type of some principal ideal. In such case,
        \begin{equation*}
            |X_{\omega}| = [\cO_0^*:\cO_n^*].
        \end{equation*}
        \item[Low ideals:] In the nonsplit case, a possible low ideal $\omega$ is the type of some principal ideal if and only if $n = de$ for some integer $d$. 
        
        In the split case, a possible low ideal $\omega = (\omega_1, \omega_2)$ is the type of some principal ideal if and only if $\omega_1 = \omega_2 = d$. 
        
        In any of these cases,
        \begin{equation*}
            |X_{\omega}| = [\cO_{n-d}^*:\cO_{n}^*] = \dfrac{[\cO_{0}^*:\cO_n^*]}{[\cO_{0}^*:\cO_{n - d}^*]}.
        \end{equation*}
    \end{description}
\end{theorem}
\begin{proof}
    This is discussed in Section 4 of \cite{MEELZETA}.
\end{proof}
\begin{example}
    To exemplify the difference between possible types and those that are actually types of ideals we show what happens for $\cO_3$. 
     \begin{figure}[ht]
    \centering
        \begin{subfigure}[t]{0.30\textwidth}
            \centering
            \includegraphics[scale = 0.6]{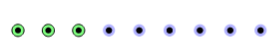}
            \subcaption{Unramified Case}
            \label{fig: unramified Impact}
        \end{subfigure}
        \begin{subfigure}[t]{0.30\textwidth}
            \centering
            \includegraphics[scale = 0.7]{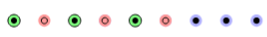}
            \subcaption{Ramified Case}
            \label{fig: unramified Impact}
        \end{subfigure}
        \begin{subfigure}[t]{0.30\textwidth}
            \centering
            \includegraphics[scale = 0.7]{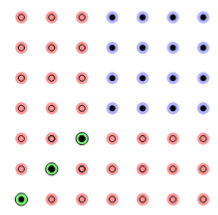}
            \subcaption{Split Case}
            \label{fig: ramified Impact}
        \end{subfigure}
       \caption{The different cases and their possible types.}
            \label{figure3: type distribution} 
    \end{figure}
    In Figure \ref{figure3: type distribution}, each dot represent a possible type. The green dots represent the low types that appear, the blue ones the high types that appear, and the red ones those who are not achieved by any ideal.
\end{example}

The indices that appear in Theorem \ref{thm:Type situation} are given by the following result.
\begin{proposition}\label{valuesindices}
The indices of the units subgroups satisfy for $n\ge 1$,
\begin{equation*}
    [\cO_0^*:\cO_n^*] = \left\{
    \begin{array}{ll}
        q^n & \mbox{in the ramified case,}\\
        (q+1)q^{n-1} & \mbox{in the unramified case,}\\
        (q-1)q^{n-1} & \mbox{in the split case.}\\
    \end{array}
    \right.
\end{equation*}
Also $[\cO_0^*:\cO_0^*] = 1$.
\end{proposition}
\begin{proof}
This is proposition 5.2 in \cite{MEELZETA}.
\end{proof}
At this stage one can use the recurrence relation and all the above information on types to compute the zeta function. This computations leads to the polynomials of Theorem \ref{thm: polynomialvalues}.

\section{Unit distribution and natural appearance of Impacted Buildings}\label{sec: Unit distribution and natural appearance of Impacted Buildings}

Notice that the polynomials found in Section \ref{sec: Explicit computations of generating functions} coincide with those arising in Theorem~\ref{thm: polynomialvalues}. Moreover, the recursive approach in each case is similar. 
The Building-theoretic approach in Section~\ref{sec: Explicit computations of generating functions} uses the recurrence relation appearing from the $n$th layer polynomial and $(n-1)$st basin polynomial to construct the $n$th basin generating function. Similarly, the ideal-theoretic perspective in Section~\ref{sec: Order polynomials for GL(2)  and their computations via ideal types} induces this recurrence relation using the traveling map.

This in particular suggests that the principal ideal zeta function \textit{is} a layer generating function evaluated at $X = q^{-s}$. The purpose of Theorem \ref{thm: principal zeta = vertex zeta} below is to prove this result. The proof of this theorem depends on a dictionary between zeta functions and impacted buildings; this is the objective of the current section. Specifically, we outline how a quadratic reduced algebra $L/K$ translates to an impacted building structure (with way out), in the usual Bruhat-Tits building of $\SL(2, K)$.

\subsection{The unit structure}

The choice of $\Delta$ amounts to a choice of basis of $L$ over $K$. Thus, it corresponds to a linear isomorphism
\begin{equation*}
        \begin{tikzcd}
             L \arrow{r} & K \times K\\
             x + y\Delta \ar[r, mapsto] & (x, y)
        \end{tikzcd}
\end{equation*}
In particular, $\cO_K-$modules of rank $2$ on each side correspond with each other; thus, we have an isomorphism of buildings. In particular, we see that we are working on $\mathcal{T}(q)$, the Bruhat-Tits building of $\SL(2, K)$.

The choice of $\Delta$, which puts coordinates on $L$, identifies the $\cO_K$-lattices 
\begin{equation*}
    O_{n} = \{x + yp^n\Delta \mid x, y\in\cO_K\}, \; n \in \mathbb{Z}.
\end{equation*}
with the vertices of the fundamental apartment (on the right hand side), corresponding to the canonical basis 
\begin{equation*}
   \{e_1 = (1, 0), e_2 = (0, 1)\},
\end{equation*}
and whose fundamental chamber (by choice) is given the edge with vertices $\cO_0$ and $\cO_1$. 

Our goal now is the locate the main sequence of orders, as well as the corresponding ideals in the building of $\SL(2, K)$. In what follows, we shall many times refer to vertices. When we do this we are always thinking on the tree of $\SL(2, K)$ under the above identification. In general, we denote the minimal polynomial of $\Delta$ by $X^2 - \tau X + \delta$ with $\tau, \delta\in \cO_K$. 

The following families of maps will allow us to understand how the structure of the tree changes in each case to accommodates to the structure of the ideals. 

\begin{definition}\label{dfn: slopmap}
    Let $n\ge 1$ be an integer. We define the map $M_n: \cO_n^*\longrightarrow \cO_K/p\cO_K$ by
    \begin{equation*}
        M_n(w + zp^{n}\Delta) = zw^{-1}.
    \end{equation*}
    We call it the \defbf{$n$-th slope map}. Notice this is well defined because $w\in\cO_K^*$.
\end{definition}

\begin{proposition}\label{prop: the slopes map}
   For $n\ge 1$, the $n$-th slope map is an epimorphism with $\ker(M_n) = \cO_{n+1}^*$.
\end{proposition}
\begin{proof}
    Let $x_1 + y_1p^n\Delta$ and $x_2 + y_2p^n\Delta$ be units of $\cO_n$. Then we have
    \begin{align*}
        M_n((x_1 + y_1p^n\Delta)(x_2 + y_2p^n\Delta))
        &= M_n(x_1x_2 +  x_1y_2p^n\Delta + y_1x_2p^n\Delta + y_1y_2p^{2n}(\tau\Delta - \delta))\\
        &= M_n((x_1x_2 - y_1y_2p^{2n}\delta) + (x_1y_2 + y_1x_2 + y_1y_2p^n\tau)p^n\Delta)\\
        &= (x_1y_2 + y_1x_2 + y_1y_2p^n\tau)(x_1x_2 - y_1y_2p^{2n}\delta)^{-1} \pmod{p}\\
        &= (x_1y_2 + y_1x_2))(x_1x_2)^{-1}\pmod{p}\\
        &= x_2^{-1}y_2 + x_1^{-1}y_1\pmod{p}\\
        &= M_n(x_1 + y_1p^n\Delta) + M_n(x_2 + y_2p^n\Delta).
    \end{align*}
This proves it is an homomorphism. Furthermore, that is is surjective follows by picking $w = 1$ and varying $z$ in a set of representatives of $\cO_K/p\cO_K$.

The kernel consists of those units with $M_n(x + yp^n\Delta)\equiv 0\pmod{p}$. That is, such that $yx^{-1}\equiv 0\pmod{p}$, which in turn implies $y\equiv 0 \pmod{p}$. This concludes the proof.
\end{proof}

\begin{corollary}\label{cor: units base step}
    Let $n\ge 1$ be an integer and $R(p)\subset \cO_K\subset\cO_L$ a fixed set of representatives of $\cO_K/p\cO_K$.  Then the units
    \begin{equation*}
        u_t^{(n)} := 1 + tp^n\Delta,\, t\in R(p),
    \end{equation*}
    are a set of coset representatives of $\cO_n^*/\cO_{n+1}^*$. In particular, the order of $\cO_n^*/\cO_{n+1}^*$ is $q$ for all $n\geq 1$.
\end{corollary}
\begin{proof}
    This is the First Isomorphism theorem applied to Proposition~\ref{prop: the slopes map}.
\end{proof}

When $n=0$, the slope map as stated only makes sense when $L/K$ is a totally ramified extension. By Proposition~\ref{valuesindices}, the quotient $\cO_0^*/\cO_1^*$ cannot inject into $\cO_K/p\cO_K$ when $L$ is unramified, for example.
At this point, we note that the distinction at $n=0$ is where the difference in structures arises. Moreover, an explicit map is not necessary; rather, we simply compute the cardinality of the quotient. To do this, Proposition~\ref{valuesindices} applies at once.

 \begin{proposition}\label{prop: indices n = 0}
     The cardinality of $\cO_0^*/\cO_1^*$ is given by 
     \begin{equation*}
         [\cO_0^*:\cO_1^*] = \begin{cases}
             q & \text{in the ramified case,}\\
             q + 1 & \text{in the unramified case,}\\
             q - 1 & \text{in the split case.}\\
         \end{cases}
     \end{equation*}
 \end{proposition}
\begin{remark}
The slope maps are used to describe jumps from $\scR_n$ to $\scR_{n+1}$; however, when $n=0,$ this descends to describing the different ways to leave the basin $\scP$. 
\end{remark}

\begin{definition}
    We extend our previous notation and denote by $u_t^{(0)}$ a fixed set of representatives of $\cO_0^*/\cO_1^*$ as $t$ varies in $R_0(p) := \{0, 1, ..., [\cO_0^*:\cO_1^*]-1\}$. 
\end{definition}

\begin{remark}
    When $L/K$ is ramified, it is sufficient to take $R_0(p) = R(p).$
\end{remark}

\begin{definition}\label{def: general multi index}
 Let $1\le d \le n$. We call a $d-$tuple,
   \begin{equation*}
       \alpha = (\alpha_{n-d}, \cdots, \alpha_{n-1}),
   \end{equation*}
   where $\alpha_i\in R(p)$ for $0 < n-d\le i \le n-1$  and $\alpha_0\in R_0(p)$, a \defbf{unit index} of length $d$ (or simply index if there is no chance for confusion). The set of unit indices of length $d$ is denoted by $I_d(n)$. In particular, $I(n) := I_n(n)$. Associated to each unit index of length $d$, say  $\alpha\in I_d(n)$, we define
    \begin{equation}\label{eq: def unit}
    u_{\alpha} := u_{\alpha_{n-d}}^{(n-d)}u_{\alpha_{n-d+1}}^{(n-d+1)}\cdots u_{\alpha_{n-1}}^{(n-1)}.
    \end{equation}  
\end{definition}

With this at hand we can describe a set of representatives of $\cO_{n-d}^*/\cO_n^*$. 
\begin{theorem}\label{thm: unit representatives low}
    Let $n\ge 0$ be an integer and $0\le d \le n$. A set of representatives of $\cO_{n-d}^*/\cO_{n}^*$ is
    \begin{equation*}
        \{u_{\alpha} \mid \alpha \in I_d(n)\}.
    \end{equation*}
\end{theorem}
\begin{proof}
    For notational simplicity, we will present the proof for $d = n$; the other cases are analogous. Let $\alpha, \beta \in I(n)$ be given and suppose $u_{\alpha}$ and $u_{\beta}$ are equivalent modulo $\cO_{n}^*$. Pick $\alpha', \beta'\in I_{n-1}(n)$ such that
    \begin{equation*}
        \alpha = (\alpha', \alpha_{n-1}),\, \beta = (\beta', \beta_{n-1}).
    \end{equation*}
    Unraveling Definition~\ref{def: general multi index} yields that $u_{\alpha} = u_{\alpha'}u_{\alpha_{n-1}}^{(n-1)},  u_{\beta} = u_{\beta'}u_{\beta_{n-1}}^{(n-1)}$. Equivalently, we have
    \begin{equation*}
          u_{\alpha}u_{\beta}^{-1} = u_{\alpha'}u_{\beta'}^{-1}u_{\alpha_{n-1}}^{(n-1)}\left(u_{\beta_{n-1}}^{(n-1)}\right)^{-1} \in \cO_n^*.
    \end{equation*}
    Furthermore, $\cO_{n}^*\subset \cO_{n-1}^*$ and $u_{\alpha_{n-1}}^{(n-1)}\left(u_{\beta_{n-1}}^{(n-1)}\right)^{-1}\in \cO_{n-1}^*$ (because each of the factors is). From these, we conclude that
    \begin{equation*}
        u_{\alpha'}u_{\beta'}^{-1} \in \cO_{n-1}^*.
    \end{equation*}
    That is, $u_{\alpha'}$ and $u_{\beta'}$ are equivalent modulo $\cO_{n-1}^*$. Repeating this process, we conclude that $u_{\alpha_0}^{(0)}$ is equivalent to $u_{\beta_0}^{(0)}$ modulo $\cO_1^*$. It follows that $\alpha = \beta$, completing the proof.
\end{proof}
\subsection{The ramified case}

As of yet, we have treated the three cases simultaneously. We now pick a particular choice of $\Delta$ in each case to specify our computations. We begin with the ramified case.

\begin{proposition}\label{prop: ramvalues of delta}
    In the ramified case $L/K$, we can choose $\Delta = \pi$ a uniformizer of $L$ satisfying the minimal polynomial
    \begin{equation*}
        X^2 - \tau X + \delta = 0,
    \end{equation*}
    with with $\valp(\tau) = t \ge 1$ and $\valp(\delta) = 1$ (i.e. $\Delta$ satisfies an Eisenstein polynomial).
\end{proposition}
\begin{proof}
    For a proof of the existence of such uniformizer as well as for a discussion on Eisenstein polynomials see \cite[Chapter II, Section 5]{LangANT}.
\end{proof}

We notice that Proposition~\ref{prop: the slopes map} applies, irrespective of our choosing $\Delta$. Thus, we set
\begin{equation*}
    u_t^{(0)} = 1 + t\pi, t\in R(p).
\end{equation*}
This choice simplifies computation. We emphasize we can only make this choice in the \textit{ramified} case. Let us now begin locating the ideals.

We begin by understanding the high ideals. Fix a positive integer $n$ and let $\omega$ be a high type; we recall this means $\omega \ge 2n$. Let $I\subset\cO_n$ be a high ideal with type $\omega$. Thus, all the generators $\alpha$ of $I$ satisfy
\begin{equation*}
    \valpi(\alpha) = \omega.
\end{equation*}
In particular, we can write $\alpha = \pi^{\omega}u$, for some unit $u\in\cO_0$ (notice the unit is not necessarily in $\cO_n)$. We get
\begin{equation*}
    I = \alpha\cO_n = \pi^{\omega}u\cO_n.
\end{equation*}
If we write $\omega = 2x + \delta$, with $\delta = 0, 1$ according to whether $\omega$ is even or odd, then
\begin{equation*}
    I = \alpha\cO_n = \pi^{\omega}u\cO_n = p^x\pi^{\delta}u\cO_n \sim \pi^{\delta}u\cO_n,
\end{equation*}
where $\sim$ represents equivalence of lattices in the building construction (i.e. homothetic lattices are equivalent). Furthermore, if $u_1$ and $u_2$ are equivalent modulo $\cO_n^*$, then we have
\begin{equation*}
    u_1\cO_n = u_2\cO_n.
\end{equation*} 
This leads to the following result.
\begin{proposition}\label{prop: high list of unit vertices}
    Let $u_1, u_2,..., u_l$ be a set of coset representatives of $\cO_0^*/\cO_n^*$. The $\cO_K$-lattices
    \begin{equation*}
        u_1\cO_n,\cdots, u_l\cO_n, u_1\pi\cO_n,\cdots, u_l\pi\cO_n,
    \end{equation*}
    represent pairwise distinct vertices of the building of $\SL(2, K)$.
\end{proposition}
\begin{proof}
    Suppose that $u_k\cO_n$ and $u_j\cO_n$ represent the same vertex in the building. By definition, there exists $\lambda\in K^*$ such that
    \begin{equation*}
        u_k\cO_n = \lambda u_j\cO_n.
    \end{equation*}
    We conclude that $\lambda\cO_n = u_j^{-1}u_k\cO_n\subset \cO_0$. In particular, $\lambda\in \cO_0$, which is to say, $\lambda\in \cO_K^*= \cO_0 \cap K^*$.

    Since $\lambda$ and $u_j^{-1}u_k$ are equivalent modulo $\cO_n,$ there exists an $o \in \cO_n$ such that $u_j^{-1}u_k = \lambda o \in \cO_0^*.$ Noting that $\cO_K^* \subset \cO_n^*,$ we find that $\lambda, o \in \cO_n^*.$ In particular,
    $$u_j^{-1}u_k \in \cO_n^*.$$
    As $u_j, u_k$ are coset representatives for $\cO_0^*/\cO_n^*$, we find they must be equal, showing that $u_1\cO_n,\cdots$, $u_l\cO_n$ are pairwise different. The same argument \say{cancelling $\pi$}, shows that
    $
        \pi u_1\cO_n,\cdots, \pi u_l\cO_n
    $
    are also pairwise different. 
    
    To conclude the proof we must now show $\pi u_j\cO_n$ is not equivalent to $ u_k\cO_n$ for any choice of $u_j, u_k$ (possibly equal).
    If they were equivalent, once more we can find $\lambda\in K^*$ with
    \begin{equation*}
        \pi u_i\cO_n = \lambda u_k\cO_n.
    \end{equation*}
    Just as before, we conclude $\lambda\in\cO_K$. We also have that there must be an element $o\in\cO_n$ with
    \begin{equation*}
        u_i\pi = \lambda u_k o.
    \end{equation*}
    Taking valuations  we get
    \begin{equation*}
        1 = \valpi(\lambda) + \valpi(o) \ge \valpi(\lambda)\ge 0
    \end{equation*}
    However, $\lambda\in \cO_K$ and thus its valuation with respect to $\pi$ is even (because of ramification). Since $\val_\pi(\lambda)\leq 1$, we conclude $\valpi(\lambda) = 0$; hence, $\lambda\in\cO_K^*$.

    The equality $\pi u_i\cO_n = \lambda u_k\cO_n$ also implies the existence of $t\in\cO_n$ with
    \begin{equation*}
        \pi u_i t = \lambda u_k,
    \end{equation*}
    which implies
    \begin{equation*}
        \pi t = \lambda u_ku_i^{-1}\in \cO_0^{*}.
    \end{equation*}
    This is impossible because the left hand side is not a unit. This contradiction concludes the proof.
\end{proof}
\begin{remark}
    We know the explicit value of $l$. Propositions \ref{valuesindices} and \ref{prop: indices n = 0} imply $ [\cO_0^*:\cO_n^*] = q^n.$
\end{remark} 

We have thus shown the following. 
\begin{corollary}\label{cor: ideals and vertices}
   Each high ideal of $\cO_n$ is equivalent to exactly one of the vertices 
   \begin{equation*}
         u_1\cO_n,\cdots, u_l\cO_n, u_1\pi\cO_n,\cdots, u_l\pi\cO_n,
  \end{equation*}
  where $u_1, u_2,..., u_l$ be a set of coset representatives of $\cO_0^*/\cO_n^*$. If $\varepsilon(I)$ is even, then $I$ is equivalent to one of
  \begin{equation*}
        u_1\cO_n,\cdots, u_l\cO_n;
\end{equation*}
    whereas is $\varepsilon(I)$ is odd, then $I$ is equivalent to one of
    \begin{equation*}
        \pi u_1\cO_n,\cdots, \pi u_l\cO_n.
\end{equation*}
\end{corollary}
In order to describe the location of the vertices $u_1\pi\cO_n,\cdots, u_l\pi\cO_n$ we must first find $\pi\cO_n$. To this end we prove the following result.

\begin{proposition}\label{prop: the other lattices description}
    We have the equality
    \begin{equation*}
        \pi\cO_n = p^{n+1}\cO_{-(n+1)} = \{x + yp^{-(n+1)}\pi \mid x, y \in\cO_K\}.
    \end{equation*}
    In particular, $\pi\cO_n$ and $\cO_{-(n+1)}$ represent the same vertex in the building and $\cO_0$ and $\pi\cO_0$ are adjacent.
\end{proposition}
\begin{proof}
     We must prove that the equality 
    \begin{equation*}
        \pi(x + yp^n\pi) = p^{n + 1}\left(a + \dfrac{b\pi}{p^{n+1}}\right)
    \end{equation*}
    determines a correspondence of pairs $(a,b)$ and $(x,y)$ over $\cO_K.$ That is, given a pair $(a,b)\in \cO_K^2$ one may solve uniquely for $(x,y)\in \cO_K^2$ and vice versa. Expanding the equality and comparing similar powers of $p$ we get
    \begin{align*}
        -\delta yp^n &= ap^{n+1},\\
        x + yp^n\tau &= b.
    \end{align*}
    Recall that $\delta = p\delta_0$ for some $\delta_0\in\cO_K^*$, the system becomes
    \begin{align*}
        -\delta_0 y&= a,\\
        x + yp^n\tau &= b.
    \end{align*}
    From this we see that $(a, b)$ and $(x, y)$ determine each other in $\cO_K^2$.
\end{proof}
To proceed we need to understand the action of multiplying in $L$ by a unit $u_{\alpha}$ that we have constructed before. To this end, consider $1 + tp^n\pi$ with $t\in R(p)$ with $n\ge 0$. We have
\begin{align*}
    (1 + tp^n\pi)(x + y\pi) 
    &= x + y\pi + xtp^n\pi + typ^n\pi^2\\
    &= x + y\pi + xtp^n\pi + typ^n(\tau\pi - \delta))\\
    &= (x - \delta typ^n) + (y + xtp^n + ytp^n\tau)\pi.
\end{align*}
The basis elements are $1$ and $\pi$, for which the above computations reduce to
\begin{align*}
    (1 + tp^n\pi)\cdot 1    &= 1 + tp^n\pi,\\
    (1 + tp^n\pi)\cdot \pi  &= - \delta tp^n + (1 + tp^n\tau)\pi.
\end{align*}
 Thus, as a matrix, the above element corresponds to
 \begin{equation*}
     \begin{pmatrix}
         1 & - \delta tp^n\\
          tp^n & 1 + tp^n\tau
     \end{pmatrix}.
 \end{equation*}
 We shall call this matrix also $u_{t}^{(n)}$. It should always be clear from the context whether we are considering it as a unit or as a matrix. Now with this computation at hand, we can prove the following result.

 \begin{proposition}\label{prop: unit action in the building}
     Let $\alpha\in I(n)$ and let $u_{\alpha}$ denote both the unit, as defined in Equation \eqref{eq: def unit} above as well as the matrix corresponding to its action in $K\times K$. Then $u_{\alpha}$ fixes the vertices $\cO_0$ and $\pi\cO_0$.
 \end{proposition}
\begin{proof}
    By construction we have that
    \begin{equation*}
        u_{\alpha} = u_{\alpha_0}^{(0)}u_{\alpha_1}^{(1)}\cdots u_{\alpha_{n-1}}^{(n-1)},
    \end{equation*}
    and we have computed above that each of the matrices on the right is of the form
    \begin{equation*}
        \begin{pmatrix}
         1 & - \delta tp^m\\
          tp^m & 1 + tp^m\tau
     \end{pmatrix},
    \end{equation*}
    for some $0\le m \le n-1$ and $t\in R(p)$. Thus, it is enough to prove the result for the matrices of this type. 
    
    Because this matrix represents a unit, there is an inverse of it, which shows it belongs to $\GL(2, \cO_K)$; equivalently, we see its determinant is a unit and its entries integral. This shows it fixes $\cO_0$. Furthermore, because $u_{\alpha}$ and $\pi$ commute (as elements of $L)$ then they commute as matrices too. Thus
    \begin{equation*}
        u_{\alpha}\left(\pi\cO_0\right) = u_{\alpha}\pi\cO_0 = \pi u_{\alpha}\cO_0 = \pi\cO_0.
    \end{equation*}
\end{proof}

We can now prove the following result; recall that the vertices $u_\alpha\cO_n$ and $u_\alpha\pi\cO_n$ are located in the standard apartment.

\begin{theorem}\label{thm: vertices at distance n1}
    The set of vertices $\{u_{\alpha} \cO_n \mid \alpha\in I(n)\}$ corresponds to the vertices at distance $n$ from $\cO_0$ and whose path connecting it to $\cO_0$ does not go through $\pi\cO_0$. Analogously, the set of vertices $\{u_{\alpha} \pi\cO_n \mid \alpha\in I(n)\}$ corresponds to the vertices at distance $n$ from $\pi\cO_0$ and whose path connecting it to $\pi\cO_0$ does not go through $\cO_0$.
\end{theorem}
\begin{proof}
    The path joining $\cO_0$ and $\cO_n$ is given precisely by the main sequence of orders. Acting on the building with $u_{\alpha}$ sends this path to the path whose endpoints are $u_{\alpha}\cO_0 = \cO_0$ and $u_{\alpha}\cO_n$. In particular, we see that the distance is preserved. Furthermore, it is not possible that this path goes through $\pi\cO_0$ because this vertex is also fixed and it does not belong to the main sequence. 

    Recall that the tree of $\SL(2, K)$ is homogeneous of degree $q + 1$. Thus, starting from $\cO_0$ there are $q^n$ vertices at distance $n$ from it that do not go through $\pi\cO_0$. We thus conclude that
    \begin{equation*}
        \{u_{\alpha} \cO_n \mid \alpha\in I(n)\}
    \end{equation*}
    correspond to \textit{all} of the vertices at distance $n$ that avoid $\pi\cO_n$. Indeed, both sets have cardinality $q^n$ and the correspondence is injective (and hence surjective too). 

    We claim that the result regarding $\{u_{\alpha} \pi\cO_n \mid \alpha\in I(n)\}$ follows from the above results by applying $\pi$ to the building. Indeed, the only fact that we must notice is that $\pi$ flips $\cO_0$ and $\pi\cO_{0}$:
    \begin{equation*}
     \pi\cdot(\pi\cO_0) = \pi^2\cO_0 = pu\cO_0 = p\cO_0 \sim \cO_0,
    \end{equation*}
    where we have used that $\pi^2 = up$ for some unit $u\in\cO_0$.
    \end{proof}
    We immediately have the following result.
    \begin{corollary}\label{cor: even and odd type land here}
    Let $n\ge 0$ be an integer. The high ideals of $\cO_n$ of even type are equivalent to vertices in $\{u_{\alpha} \cO_n \mid \alpha\in I(n)\}$, while those of odd type are equivalent to vertices in $\{u_{\alpha}\pi \cO_n \mid \alpha\in I(n)\}$. 
\end{corollary}
\begin{proof}
    Such ideals are classified in Corollary~\ref{cor: ideals and vertices} and the association is given by Theorem~\ref{thm: vertices at distance n1}.

\end{proof}

Now that we understand where the high ideals land, let us locate the low ideals. Recall that a principal ideal $I\subseteq \cO_n$ is a \textit{low ideal} if its type $\omega$ satisfies $\omega \le 2n - 1$. Furthermore, by Theorem \ref{thm:Type situation} that the low ideals $I$ have types $0, 2, 4, ..., 2(n - 1)$. 

Let $I$ be a principal low ideal of type $2d$. Furthermore, let $x + yp^n\pi$ be a generator of $I$. One can verify (see, for example \cite[Propositions 4.12, Proposition 4.14]{MEELZETA}) that 
$
    \valp(x) = d.
$
Then
\begin{equation*}
    I = (x + yp^n\pi)\cO_n = (p^du + yp^n\pi)\cO_n = p^d(u + yp^{n-d}\pi)\cO_n \sim (u + yp^{n-d}\pi)\cO_n.
\end{equation*}
Notice that in the above, $u + yp^{n-d}\pi\in\cO_{n-d}^*$.  We realize that the vertices that represent the low ideals are
\begin{equation*}
    u_{\alpha}\cO_n, \; \alpha\in I_d(n),
\end{equation*}
which is to say,
\begin{equation*}
    u_{\alpha_{n-d}}^{(n-d)}u_{\alpha_{n-d+1}}^{(n-d+1)}\cdots u_{\alpha_{n-1}}^{(n-1)}\cO_n, \; \alpha\in I_d(n).
\end{equation*}
We now have the following result.
\begin{theorem}\label{thm: vertices at distance n}
    The set of vertices $\{u_{\alpha} \cO_n \mid \alpha\in I_d(n)\}$ correspond to the vertices at distance $d$ from $\cO_{n-d}$ and whose path connecting it to $\cO_{n-d}$ does not go through $\cO_{n-d-1}$. 
\end{theorem}
\begin{proof}
    Notice that all the units
    \begin{equation*}
        u_{\alpha_{n-d}}^{(n-d)}, u_{\alpha_{n-d+1}}^{(n-d+1)}, \cdots, u_{\alpha_{n-1}}^{(n-1)}
    \end{equation*}
    belong to $\cO_{n-d}^*$. As a consequence,
    \begin{equation*}
        u_{\alpha_i
        ^{(i)}}\cO_{n-d} = \cO_{n-d}, \; n - d \le i \le n - 1.
    \end{equation*}
We have seen that $u_{\alpha_i^{(i)}}$ stabilizes $\cO_{0}$. Thus, under the action of $u_{\alpha}$ the whole path from $\cO_0$ to $\cO_{n-d}$ is fixed. On the other hand, the path from $\cO_{n-d}$ to $\cO_{n}$ goes under the same element to the path from $\cO_{n-d}$ to $u_{\alpha}\cO_{n}$. However, this path cannot touch $\cO_{n-d-1}$ because this vertex is fixed. 

One concludes that the vertices that can be reached are those at length $d$ from $\cO_{n-d}$, as the distance is preserved, but that avoid $\cO_{n-d-1}$. All such vertices are attained by a unique unit action, as the cardinalities of the sets of coset representatives and vertices are equal. 
\end{proof}

Almost all of our discussion is summarized in the following result.

\begin{theorem}\label{thm: Ideals Ram Case}
    Let $n\ge 0$. The (principal rank $2$) ideals of $\cO_n$ are equivalent to the vertices at distance $n$ from the set $\{\cO_0, \pi\cO_n\}$.
\end{theorem}
\begin{remark}
    The above statement is independent of the type of $I.$ A vertex can represent ideals of both high and low type, or only of high type. So, our work is yet to be complete. The description of what types are represented by a single vertex will be answered in Theorem \ref{thm:structure of types at x}.
\end{remark}

\subsection{The unramified case }

The analysis of the remaining two cases follows the same structure as the ramified case. A common theme of our work is that much of the theory developed overlaps between cases; in what follows proofs that follow such a pattern will be omitted.

\begin{proposition}\label{prop: unramvalues of delta}
    If $L/K$ is an unramified quadratic extension, we can pick $\Delta = \sqrt{\epsilon}$,
    where $\epsilon$ is a nonsquare integral unit in $\cO_K^*$.
\end{proposition}

We once again first locate the high ideals. Fix $\omega \geq n$ so that if an ideal $I\subset \cO_n$ is of (high) type $\omega$, then 
\begin{equation*}
    I = up^{\omega}\cO_n \sim u\cO_n,
\end{equation*}
where $u\in\cO_0^*$. We once again have the following.

\begin{proposition}\label{prop: unramhigh list of unit vertices}
    Let $u_1, u_2,..., u_l$ be a set of coset representatives of $\cO_0^*/\cO_n^*$. Then the $\cO_K$-lattices
    \begin{equation*}
        u_1\cO_{n}, u_2\cO_n,\cdots, u_l\cO_n
    \end{equation*}
    represent pairwise distinct vertices of the building of $\SL(2, K)$.
\end{proposition}
\begin{proof}
    This is proven analogously to Proposition \ref{prop: high list of unit vertices}.
\end{proof}

A basis for $\cO_L$ is $\{1, \sqrt{\epsilon}\}$; hence, for $n\ge 1$, we may consider 
\begin{align*}
    (1 + tp^n\sqrt{\epsilon}) &= 1 + tp^n\sqrt{\epsilon}\\
    (1 + tp^n\sqrt{\epsilon})\sqrt{\epsilon} &= tp^n\epsilon + \sqrt{\epsilon}.
\end{align*}
Using this, we may encode this action as the matrix \begin{equation}\label{matrixActionUnramified}
    u_t^{(n)}=\begin{pmatrix}
        1 & tp^n\epsilon \\
        tp^n & 1
    \end{pmatrix}.
\end{equation}
We can also consider the action by $u_t^{(0)}$ for $t\in R_0(p)$ as some matrix, which we do not specify.

 \begin{proposition}\label{prop: unramunit action in the building}
     Let $\alpha\in I(n)$ and let $u_{\alpha}$ denote both the unit, as defined in equation \eqref{eq: def unit} above, as well as the matrix corresponding to its action in $K\times K$. Then $u_{\alpha}$ stabilizes the vertex $\cO_0$.
 \end{proposition}
\begin{proof}
    Every involved matrix has integral entries and unit determinant, thus it represents an element of $\GL(2, \cO_K)$. This also holds for $n = 0$, even if we do not know its matrix. The result follows.
\end{proof}

We now turn our attention to the low types. Theorem \ref{thm:Type situation} guarantees that the low ideals $I$ have type $0, 1,\ldots, n-1$. Let us choose $0\leq d < n$ and an ideal of type $d.$ Just as in the ramified case, if $x+yp^n\sqrt{\epsilon}$ is a generator of $I$, then $\valp(x) = d$. Thus $x=up^d,$ where $u\in \mathcal{O}_K^*$. From this it follows that
\begin{equation*}
    I = (x+yp^n\sqrt{\epsilon})\cO_n 
    = p^n(u+yp^{n-d}\sqrt{\epsilon})\cO_n 
    \sim (u+yp^{n-d}\sqrt{\epsilon})\cO_n,
\end{equation*}
with $u+yp^{n-d}\sqrt{\epsilon} \in \cO_{n-d}^*$. We thus have the following pair of results.

\begin{theorem}\label{thm: unramvertices at distance n} 
    Given $u_1,u_2\in\cO_{n-d}^*$, we have that $u_1\cO_{n} \sim u_2\cO_n$ if and only if $u_1u_2^{-1} \in \cO_n^*.$ In particular, if $u_1,...,u_l$ is a set of representatives of $\cO_{n-d}^*/\cO_n^*$, then the $\cO_K$-lattices 
    \begin{equation*}
        u_1\cO_{n}, u_2\cO_n,\cdots, u_l\cO_n
    \end{equation*}
    represent pairwise disjoint vertices of the building of $\SL(2, K)$.
\end{theorem}

\begin{theorem}\label{thm: low ideals UNRAM}
    The set of vertices $\{u_\alpha\cO_n: \alpha\in I_d(n)\}$ are exactly those at distance $d$ from the vertex $\cO_{n-d}$, excluding those for who the geodesic connecting these two vertices intersects the vertex $\cO_{n-d-1}$.
\end{theorem}
\begin{proof}
    One may reproduce the proof of Theorem \ref{thm: vertices at distance n}.
\end{proof}

Our discussion leads to the following theorem in this case. Once more, we emphasize that this description forgets the information associated with types.

\begin{theorem}\label{thm: Ideals Unram Case}
    Let $n\ge 0$. The (principal rank $2$) ideals of $\cO_n$ are equivalent to the vertices at distance $n$ from the set $\{\cO_0\}$.
\end{theorem}

\subsection{The split case}

When considering the split case, we will discuss high and low ideals simultaneously. Let us begin as usual, by setting $\Delta$.

\begin{proposition}\label{prop: splitvalues of delta}
    In the split case $L = K\times K$ we can pick $\Delta = (1, p)$. 
\end{proposition}
\begin{proof}
    We notice that $X^2 - (p+1)X + p = 0$ has roots $1$ and $p$; moreover, the determinant of the matrix
    $$\begin{pmatrix}
        1 & 1 \\ 1 & p
    \end{pmatrix}$$ lies in $\cO_K^*$. Hence, we conclude that it generates $\cO_K$, concluding our proof.
\end{proof}

Notice that in this case, a basis of $\cO_n$ is $(1, 1)$ and $p^n(1, p)$. We define for each integer $m$
\begin{equation*}
    \widetilde{\cO_m} = (1, 0)\cdot\cO_k \oplus (0, p^m)\cdot\cO_k.
\end{equation*}

\begin{proposition}\label{prop: unit action split case}
    Let $n\ge 0$. Then we have
    \begin{enumerate}
        \item The unit $u_t^{(n)}$ fixes the vertices $\cO_0$ for all $n\ge 0$.
        \item The unit $u_{\alpha}\in I_d(n)$ fixes the vertex $\cO_{n-d}$.
        \item The unit $u_t^{(n)}$ fixes the vertices $\widetilde{\cO_m}$ for all $m$.
        \item The unit $u_{\alpha}\in I_d(n)$ fixes $\widetilde{\cO_m}$ for all $m$.
    \end{enumerate}
\end{proposition}
\begin{proof}
    By definition, $u_t^{(n)}\cO_n^*\subset \cO_0^*$. Thus as a matrix $u_l^{(n)}$ has integral entries and unit determinant. Thus it lies in $\GL(2, \cO_K)$ and hence fixes $\cO_0$. This proves (1).

    By definition, we have that 
    \begin{equation}
    u_{\alpha} = u_{\alpha_{n-d}}^{(n-d)}u_{\alpha_{n-d+1}}^{(n-d+1)}\cdots u_{\alpha_{n-1}}^{(n-1)}.
    \end{equation}
    Each of the units on the right belongs to $\cO_{n-d}^*$. Thus they fix it as a lattice and thus as a vertex. This proves (2).

    To prove (3), notice that
    \begin{equation*}
        \widetilde{\cO}_m = (1, p^m)\cdot\cO_0.
    \end{equation*}
    As a consequence,
    \begin{equation*}
        u_t^{(n)}\cdot \widetilde{\cO}_m =  u_t^{(n)}\cdot (1, p^m)\cdot\cO_0 = (1, p^m)\cdot u_t^{(n)}\cdot\cO_0  = (1, p^m)\cdot\cO_0 = \widetilde{\cO}_m.
    \end{equation*}
    This proves (3) and, because the units in $I_d(n)$ are compositions of the units covered by (3), we also have proven (4).
\end{proof}

We now locate the ideals in the building. Recall that our ideals are always principal and have rank $2$ as $\cO_K$-modules. If $I$ is such an ideal and $\alpha = (\alpha_1, \alpha_2)$ one of its generators, we have its type is given by
\begin{equation*}
    \varepsilon(I) = (\valp(\alpha_1), \valp(\alpha_2)).
\end{equation*}
Theorem~\ref{thm:Type situation} determines if $I$ is of high type or low type.

We compute
\begin{equation*}
    I = (\alpha_1, \alpha_2)\cO_n = (p^{m_1}u_1, p^{m_2}u_2)\cO_n = (1, p^{m_2-m_1}u_2u_1^{-1})p^{m_1}u_1\cO_n \sim (1, p^{m_2-m_1}u_2u_1^{-1})\cO_n,
\end{equation*}
where $m_i = \valp(\alpha_i), i = 1, 2$ and $u_1, u_2$ are units in $\cO_K^*$. If we denote $m = m_2 - m_1$, this becomes
\begin{equation*}
    I \sim (1, u_2u_1^{-1})\cdot (1, p^m)\cdot\cO_n.
\end{equation*}
However, $(1, u_2u_1^{-1})$ is equivalent to exactly one of the units in $I(n)$. Denote this unit by $u$; then, we have
\begin{equation*}
     (1, u_2u_1^{-1})\cdot\cO_n. = u\cdot \cO_n,
\end{equation*}
because their difference is in $\cO_n^*$ which fixes $\cO_n$. 
\begin{proposition}
    All (principal, rank 2) ideals $I\subseteq\cO_n$ are equivalent to a vertex given by
    \begin{equation*}
        u_{\alpha}\cdot (1, p^m)\cdot \cO_n,
    \end{equation*}
    for some $\alpha\in I(n)$ and $m\in\mathbb{Z}$. It represents a low ideal if we can take $\alpha\in I_d(n)$ for some $1\le d \le n-1$ and $m = 0$.
\end{proposition}
\begin{proof}
    It remains to prove the result about low ideals. Indeed, if $I$ is a low ideal, then $m_1 = m_2$, which proves $m = 0$. Furthermore, being a low ideal guarantees that $u_2u_1^{-1}\in\cO_{n-d}^*$ for some $1\le d \le n-1$; and thus, we can use the indices $I_d(n)$ (or, pick some continuous chain of units, starting from the left, for the index in $I(n)$ to be $1$).
\end{proof}

To conclude our description, we give the following theorem.

\begin{theorem}\label{thm: location of ideals split case}
    The set of vertices
    \begin{equation*}
        \{u_{\alpha}\cdot (1, p^m)\cdot \cO_n:\, \alpha\in I(n),\, m\in \Z\},
    \end{equation*}
    represents, exactly once, the set of vertices of the building at distance exactly $n$ from the apartment with vertices
    \begin{equation*}
        \widetilde{\cO}_m, m\in\mathbb{Z}.
    \end{equation*}
    Furthermore, the set of vertices
    \begin{equation*}
        u_{\alpha} \cdot \cO_n,
    \end{equation*}
     for some $\alpha\in I_d(n)$ with $1\le d \le n-1$ represents, exactly once, the set of vertices of the building at distance $d$ from $\cO_{n-d}$ and whose geodesic to this point doesn't go through $\cO_{n-d-1}$. Only these latter vertices can represent low ideals.
\end{theorem}
\begin{proof}
   Let $\Gamma$ be the geodesic containing $\cO_0$ and $\cO_n$. Notice that $u_\alpha$ stabilizes $\cO_0,$ so that $d(\cO_0, u_\alpha\cO_n)=n$. Moreover, $u_\alpha$ stabilizer $\widetilde{\cO}_m$, for all $m$; thus we conclude that $\Gamma$ and $\widetilde{\cO}_m$ only share a vertex when $m=0.$ It follows that
\begin{equation*}
    \{u_{\alpha}\cO_0 \mid \alpha\in I(n)\},
\end{equation*}
consists of the vertices at distance exactly $n$ from $\cO_0$ except from $\widetilde{\cO}_n$ and $\widetilde{\cO}_{-n}$. 
Now, noting that $(1,p^m)\cdot \cO_0 = \widetilde{\cO}_m$ and $u_\alpha\widetilde{\cO}_m = \widetilde{\cO}_m$, we find that 
\begin{equation}\label{eq: displacement of varying m}
    (1, p^m)\cdot  \{u_{\alpha}\cO_0 \mid \alpha\in I(n)\}
\end{equation}
consists of all those vertices at distance exactly $n$ from $\widetilde{\cO}_m$ except from $\widetilde{\cO}_{m + n}$ and $\widetilde{\cO}_{m-n}$. 
Moreover, notice that the sets in \eqref{eq: displacement of varying m} are disjoint; the first statement follows as each $\widetilde{\cO}_m$ appears in such a set. 

The second claim of the theorem follows similarly. This time there is no $m$, thus the vertices are located among those at distance $n$ from $\cO_0$. However, as our index is taken in $I_d(n)$, we have that $u_{\alpha}$ fixes $\cO_0, \cO_1,..., \cO_{n-d}$ by Proposition~\ref{prop: unit action split case}. As a consequence, the geoedesic towards $\cO_{n-d}$ has length $d$. This geodesic avoids $\cO_{n-d-1}$ because this vertex is fixed by $u_{\alpha}$. We have already seen above that these vertices admit low ideal representatives. This concludes the proof.
\end{proof}

Finally, our (almost complete) summary of the situation is the following.

\begin{theorem}\label{thm: Ideals split Case}
    Let $n\ge 0$. The (principal rank $2$) ideals of $\cO_n$ are equivalent to the vertices at distance $n$ from the set $\{\widetilde{\cO}_m : m\in\mathbb{Z}\}$.
\end{theorem}

\subsection{The appearance of the Impacted Buildings}

Up to this point, we have completely described which vertices of the building are represented by the low and high ideals of a given fixed order $\cO_n$. In the three cases we have found that there is a subset of vertices, represented by some of the ideals of $\cO_n$, which can be connected by actions of \say{unit paths} on the building to the vertices of an specific set of vertices. We recall that in Definition \ref{def: impacted buildings cases defs} we have introduced three impacted buildings  $\scB_R$, $\scB_U$ and $\scB_S$. Finally, we can provide the first connection between the two perspectives. 

\begin{theorem}\label{thm: ideals in the building}
    Let $L/K$ be a reduced $K$-algebra of dimension $2$. It is either ramified, unramified or split. If $\scB$ is the corresponding impacted building, that is $\scB_R$, $\scB_U$ or $\scB_S$ respectively, then a vertex $v$ of $\scB$ represents a (principal, rank $2$) ideal of $\cO_n$ if and only if $v\in\scR_n$.
\end{theorem}
\begin{proof}
    This has been shown in a case-by-case basis. If $L/K$ is ramified, this theorem specializes to Theorem~\ref{thm: Ideals Ram Case}, if $L/K$ is unramified, this theorem specializes to Theorem~\ref{thm: Ideals Unram Case}, and if $L/K$ is split, this theorem specializes to Theorem~\ref{thm: Ideals split Case}.

\end{proof}

\section{Discussion and proof of the main theorem}\label{sec: Discussion and proof of the main theorem}

So far, we have characterized which exact vertices in the building represent ideals from $\cO_n$, for any given $n$, for all our cases of $L/K$. However, a single vertex can represent several principal ideals with different types. We must describe with precision which types are represented at each vertex in order to be able to make explicit computations.

\subsection{Types represented by a vertex}

We begin with a general definition, although we use it only in our three quadratic cases.

\begin{definition}\label{def: ramification vector}
    Let $L = L_1\times\cdots\times L_g$ be reduced finite dimensional $K$-algebra. Thus, each of $L_1,\cdots,L_g$ are field extensions of $K$ and we let $e_i$ be the ramification degree of $L_i/K$. We shall call the vector
    \begin{equation*}
        \vec{e} = \vec{e}(L/K) := (e_1,..., e_g)
    \end{equation*}
    the \defbf{translation vector} of the extension $L/K$. We also define 
    \begin{equation*}
        \vec{f} = \vec{f}(L/K) := (f_1,..., f_g)
    \end{equation*}
    and call it the \defbf{inertia vector} of the extension $L/K$.
\end{definition}
\begin{remark}
    The choice for the adjective \say{translation}, as opposed to, say, \say{ramification}, will become clear presently.
\end{remark}

In particular, for our three quadratic cases we have
    \begin{equation*}
        \vec{e}(L/K) = \begin{cases}
            1 &  \mbox{if } L/K \mbox{ is an unramified field extension},\\
            2 &  \mbox{if } L/K \mbox{ is a ramified field extension},\\
            (1, 1) & \mbox{if } L = K\times K.\\
        \end{cases}
    \end{equation*}
    and
    \begin{equation*}
        \vec{f}(L/K) = \begin{cases}
            2 &  \mbox{if } L/K \mbox{ is an unramified field extension},\\
            1 &  \mbox{if } L/K \mbox{ is a ramified field extension},\\
            (1, 1) & \mbox{if } L = K\times K.\\
        \end{cases}
    \end{equation*}

\begin{definition}
    Let $x\in\scR_n$ be a vertex. We define 
    \begin{equation*}
        \varepsilon(x) = \{\varepsilon(I) \mid I\sim x, I\subseteq\cO_n, I \mbox{ principal }\};
    \end{equation*}
    that is, $\varepsilon(x)$ is the set of types from principal ideals of $\cO_n$ represented by $x$. 
\end{definition}

We begin by describing the structure of $\varepsilon(x)$. 
\begin{lemma}\label{lem: congruence of types}
    Let $I$  and $J$ two (principal, rank 2) ideals of $\cO_n$ with $I\sim J$. Then,
    \begin{equation*}
        \varepsilon(I) \equiv \varepsilon(J) \pmod{\vec{e}}.
    \end{equation*}
    To be precise, what we mean here is that the difference $\varepsilon(I) - \varepsilon(J)$ is an integral multiple of the vector $\vec{e}$.  
\end{lemma}
\begin{proof}
    Since $I$ and $J$ are equivalent in the tree, there exists $\lambda\in K^*$ such that
    \begin{equation*}
        I = \lambda J.
    \end{equation*}
    Thus, the representatives of $I$ (as a principal ideal) are shifts of those of $J$ by $\lambda$. That is,
    \begin{equation*}
        R(I) = \lambda\cdot R(J).
    \end{equation*}
    By definition of the type of an ideal, it is the constant vector of valuations of its generators (as an ideal). In particular, taking valuations entrywise, we find
    \begin{equation*}
        \varepsilon(I) = (\val_{L_i}(\lambda))_{i = 1,..., g} + \varepsilon(J).
    \end{equation*}
    Here $g = 1$ in the field extension case and $g = 2$ in the split case. Notice that in each case $(\val_{L_i}(\lambda))_{i = 1,..., g}$ is an integer multiple of $\vec{e}$. Indeed, this vector is
    \begin{equation*}
        (\val_{L_i}(\lambda))_{i = 1,..., g} = \begin{cases}
            \valp(\lambda) &  \mbox{if } L/K \mbox{ is an unramified field extension},\\
            \valpi(\lambda)\ &  \mbox{if } L/K \mbox{ is a ramified field extension},\\
            (\valp(\lambda), \valp(\lambda)) & \mbox{if } L = K\times K.\\
        \end{cases}
    \end{equation*}
    We only need to notice that in the ramified case, $\valpi(\lambda)$ is even because $\lambda\in K^*$ and that, in the split case, we obtain $(\valp(\lambda), \valp(\lambda))$ because $\lambda$ acts by entrywise multiplication. This concludes the proof.
\end{proof}
The following fact follows easily from what we have done in the above proof; we isolate it due to its relevance.
\begin{lemma}\label{lem: equal type an vertex implies equal}
    Let $I$  and $J$ two (principal, rank 2) ideals of $\cO_n$ with $I\sim J$. Then $I = J$ if and only if $\varepsilon(I) = \varepsilon(J)$
\end{lemma}
\begin{proof}
    Only the backwards implication needs to be proven. Indeed, if $\varepsilon(I) = \varepsilon(J)$, then by the proof of Lemma \ref{lem: congruence of types}, we have that $\lambda$ is a unit in $\cO_K^*$. Since multiplication by units does not change the ideal, we conclude that
    $
        I = \lambda J = J,
    $ as desired
\end{proof}

\begin{definition}
    Let $I$ and $J$ be two (principal, rank 2) ideals of $\cO_n$. We say $I\le J$ if and only if \begin{itemize}
        \item $I\sim J,$ and 
        \item $\varepsilon(J) - \varepsilon(I)$ has nonnegative entries
    \end{itemize}
\end{definition}

\begin{proposition}\label{prop: source type}
    Let $I$ and $J$ be two (principal, rank $2$) ideals of $\cO_n$. The following statements hold. \begin{description}
        \item[(a)] The relation $\le$ defines a partial order on the set of (principal, rank 2) ideals of $\cO_n$.
        \item[(b)] Let $x\in\scR_n$ be a vertex. When restricted to the ideals $I\sim x$, $\leq$ becomes a total order.
        \item[(c)] Let $x\in\scR_n$ be a vertex. Then $\leq$ has a unique minimum element.
    \end{description}
\end{proposition}
\begin{proof}
    That these conditions constitute a reflexive and transitive relation is straightforward. Lemma~\ref{lem: equal type an vertex implies equal} implies it is antisymmetric. This proves (a)

    If we now restrict to those ideals that are equivalent to a fixed vertex, then, by Lemma \ref{lem: congruence of types} and its proof, we conclude we can always compare them. Indeed, its difference is an integer multiple of $\vec{e}$, which is a vector with positive entries. Thus one of 
    \begin{equation*}
         \varepsilon(J) - \varepsilon(I),\,  \varepsilon(I) - \varepsilon(J),
    \end{equation*}
    has nonnegative entries. This proves $(b)$.

    Finally, to prove that a minimum exist, we notice that given any ideal $I\sim x$, the types of the other ideals that are also equivalent to $x$ lie in the set
    \begin{equation}\label{eq: potential types}
        \{\varepsilon(I) + k\vec{e},\, k\in\mathbb{Z}\},
    \end{equation}
    by Lemma \ref{lem: congruence of types}. However, because the entries of ideal types are nonnegative, the actual ideal types that do appear in \eqref{eq: potential types} have a lower bound. 
    
    Furthermore, by definition, the sign of $k$ determines whether these types are smaller or bigger (or equal if $k = 0$). In particular,
    \begin{equation*}
        \varepsilon(I) + k\vec{e} \longrightarrow k,
    \end{equation*}
    embeds this total order into the one of $\mathbb{Z}$. Since there is a lower bound, $(c)$ follows.
\end{proof}

With the above results, we are now allowed to give the following definition and structural result.

\begin{definition}\label{def: source}
    Let $x\in\scR_n$. The minimum type of the total order proven in part (b) and (c) of Proposition \ref{prop: source type} will be called the \defbf{type source of vertex $x$} or simply \defbf{the source at $x$}. We will denote it by $\omega(x)$.
\end{definition}
\begin{theorem}\label{thm:structure of types at x}
    Let $x\in\scR_n$ be a vertex with source $\omega(x)$, then
    \begin{equation*}
        \varepsilon(x) = \{\omega(x) + k\vec{e} \mid k = 0, 1, 2, ...\}.
    \end{equation*}
    In other words, if $\omega(x)$ is the \say{minimum} type represented by the vertex $x$, then $x$ represents ideals of any type \say{bigger} than $\omega(x)$ under $\leq$.
\end{theorem}
\begin{proof}
    Let $I$ be an ideal that achieves the minimum at $x$. Notice that multiplication by $p$ produces another principal ideal, also represented by $x$, and with
    \begin{equation*}
        \varepsilon(pI) = \vec{e} + \varepsilon(I) = \vec{e} + \omega(x).
    \end{equation*}
    With this we see that by successive multiplication by $p$, we can produce all types in $\{\omega(x) + k\vec{e} \mid k = 0, 1, 2, ...\}$. Thus
    \begin{equation*}
        \varepsilon(x) \supseteq \{\omega(x) + k\vec{e} \mid k = 0, 1, 2, ...\}.
    \end{equation*}
    On the other hand, Lemma \ref{lem: congruence of types} and Definition~\ref{def: source}, prove 
    \begin{equation*}
        \varepsilon(x) \subseteq \{\omega(x) + k\vec{e} \mid k = 0, 1, 2, ...\},
    \end{equation*}
    concluding the proof.
\end{proof}

\subsection{The equality of zeta functions}

In order to make our final result easy to explain we need one last definition, which is justified by our work in the previous section.

\begin{definition}\label{def: coordinates}
    Let $I\subseteq\cO_n$ be a (principal, rank 2) ideal. We define its \defbf{vertex-type coordinates} as the pair $(x, \varepsilon(I))$ where $x\in\scR_n$ is the vertex such that $I\sim x$.
\end{definition}

The equivalence of generating functions we desire to prove essentially follows from reading the set of coordinates in two different ways. However, before we can do this, we must connect the source at $x$ with some property of $x$ as a vertex of the impacted building. The next lemma gives this connection.

\begin{lemma}\label{lem: contribution is distance}
    Let $x\in\scR_n$, then
    \begin{equation}\label{eq: contribution is distance}
        d(x, \cO_n) = c(\omega(x)),
    \end{equation}
    where $c(\omega(x))$ is the type \textit{contribution} of $\omega(x)$.
\end{lemma}
\begin{proof}
   We have seen that for each case all the vertices in $\scR_n$ represent high ideals. However, not all of them represent low ideals. Fortunately, the latter have been described in Theorem \ref{thm: vertices at distance n}, Theorem \ref{thm: low ideals UNRAM}, and Theorem \ref{thm: location of ideals split case}. This description gives that reaching this vertex $x$ from $\cO_n$ requires to go first to $\cO_{n-d}$ and from there to $x$. 
   
   Let $x$ be a vertex that represents low ideals and let $I$ be a low ideal represented by it with the lowest contribution; thus, $\omega(x) = \omega(I)$. Furthermore, we know this $I$ is unique. Concretely,
   \begin{equation*}
       I \sim u_{\alpha}\cdot\cO_n = x,
   \end{equation*}
   for some $\alpha\in I_d(n)$ with $d$ minimal, and in particular, smaller than $n$. In this situation, the two paths mentioned in the above paragraph are disjoint (otherwise we could decrease $d$ because the paths would also share $\cO_{n-(d-1)}$). Consequently,
   \begin{equation*}
       d(V, \cO_n) = d(x, \cO_{n-d}) + d(\cO_{n-d}, \cO_n) = d + d = 2d.
   \end{equation*}
    However, $I$ is constructed in a very specific way from this information. Concretely,
    \begin{equation*}
        I = {p}^{d}u_{\alpha}\cO_{n}
    \end{equation*}
    Notice that this ideal is indeed a principal ideal \textit{of $\cO_n$} because $u_{\alpha}\in\cO_{n-d}^*$ and thus the corresponding power of $p$ by which it gets multiplied guarantees that the \say{representative} is an element of $\cO_n$. We now compute the type contribution of this ideal. We have
    \begin{equation*}
        \omega(x) = \omega(I) = \begin{cases}
            f(de) & \mbox{in the nonsplit case,}\\
            d + d & \mbox{in the split case.}
        \end{cases}
    \end{equation*}
    Using that $ef = 2$ we get that in all cases
    \begin{equation*}
        \omega(x) = 2d = d(x, \cO_n),
    \end{equation*}
    as desired. This settles the situation where $x$ represents low ideals. 

    Let us now take $x$ to be a vertex that \textit{does not} represent low ideals. Let $I$ be an ideal, necessarily high, that is represented by $x$ and with the lowest possible type contribution. We know that
    \begin{equation}\label{eqn: formula for I}
        I = u_{\alpha}{\pi}^{\varepsilon(I)}\cO_n,  \;\;\alpha\in I(n)
    \end{equation}
   where in the split case ${\pi}^{t_n + \varepsilon(I)} = (p^{n + \alpha_1}, p^{n + \alpha_2})$. 
   
   In all these cases, the geodesic from $\cO_n$ to $x$ goes through $\cO_0$. Specifically, this geodesic has three parts: one from $\cO_n$ to $\cO_0$, another from $\cO_0$ to ${\pi}^{\varepsilon(I)}\cO_0$, and a final one from ${\pi}^{\varepsilon(I)}\cO_0$ to $x$. The first and last segments both have length $n$. The middle segment occurs entirely within the basin and its length is determined by ${\pi}^{\varepsilon(I)}$. As we are dealing with high ideals, the type has to be \say{above} the ideal threshold. 
   
   In the field extension case this means $\varepsilon(I)\ge ne$, while in the split case it means each entry is at least $n$. Writing this as $\varepsilon(I) = t_n + (m_1, m_2)$ in the split case and as $\varepsilon(I) = t_n + m$ otherwise, we see
   \begin{equation}\label{eqn: formula for omega vertex}
       \omega(x) = \omega(I) = \begin{cases}
           f(ne + m) \\
           n + m_1 + n + m_2 
       \end{cases}
       = 2n + \begin{cases}
           fm & \mbox{in the nonsplit case}\\
            m_1 + m_2 & \mbox{in the split case}.
       \end{cases}
   \end{equation}
   To conclude we must use the fact that $I$ produces the minimal type at the vertex $x$. Notice that equation \eqref{eqn: formula for I} for $I$ represents the same vertex whenever we can \say{factor powers of $p$} from $\pi^{m}$ or $\pi^{(m_1, m_2)}$. For example, in the unramified case, if $m \neq 0$, then as vertices
   \begin{equation}
        u_{\alpha}p^n\cdot{p}^{m}\cO_n \sim u_{\alpha}p^n\cdot{p}^{m-1}\cO_n,
    \end{equation}
    but as ideals the right hand one has smallest type. This contradicts the minimality of the type. Hence, we necessarily have $m = 0$.

    In general, by Proposition \ref{prop: source type}, we know that any displacement in the type by $\vec{e}$ produces the same vertex. In particular, we must be unable to subtract $\vec{e}$. This means in the ramified case that $m \in \{0,1\}$, in the unramified case that $m = 0$, and in the split case that $m \in \{ (0, m_2), (m_1, 0)\}$.

    In the unramified case, where $m = 0$, equation \eqref{eqn: formula for omega vertex} implies
    \begin{equation*}
        \omega(x) = 2n = d(\cO_n, \cO_0) + d(\cO_0, x).
    \end{equation*}
    In the ramified case, the geodesic from $\cO_n$ to $x$ has length
    \begin{equation*}
        d(x, \cO_n) = d(\cO_n, \cO_0) + d(\cO_0, {\pi}^{\varepsilon(I)}\cO_0) + d({\pi}^{\varepsilon(I)}\cO_0, x) = 2n + d(\cO_0, \pi^m\cO_0).
    \end{equation*}
    We finally notice that $d(\cO_0, \pi^m\cO_0) \in\{0,1\}$ according to whether $m$ is even or odd. The latter condition is equivalent to whether $I$ is of even or odd type, which by comparing values yields
    \begin{equation*}
        d(x, \cO_n)  = \omega(x).
    \end{equation*}
    Finally, we deal with the split case. The arguments are the same as above. The term \say{+ 2n} in the contribution matches the distances $d(\cO_n, \cO_0)$ and $d({\pi}^{\varepsilon(I)}\cO_0, x)$. Thus, we only need to explain the term $d(\cO_0, {\pi}^{\varepsilon(I)}\cO_0)$. In our current case, since one of $m_1$ or $m_2$ is zero, what we have is that ${\pi}^{\varepsilon(I)}\cO_0 = \widetilde{\cO_{-m_1}}$ or $\widetilde{\cO_{m_2}}$, according to whether $m_1\neq 0$ or $m_2\neq 0$, respectively. Thus
    \begin{equation*}
        d(\cO_0, {\pi}^{\varepsilon(I)}\cO_0) = m_i.
    \end{equation*}
    We see this matches the missing term from the sum giving $\omega(x)$. Thus, once more
    \begin{equation*}
        d(x, \cO_n)  = \omega(x).
    \end{equation*}
    This concludes the proof.
\end{proof}

We finally prove our main result.

\begin{theorem}\label{thm: principal zeta = vertex zeta}
    Let $L/K$ be one of the above considered extensions (i.e. unramified, ramified, split extensions) and $\scB$ the corresponding impacted building (i.e. $\scB_U, \scB_R$, $\scB_S$). Then we have the equality
    \begin{equation}\label{eq: principal zeta = vertex zeta}
        \zeta_{\cO_n}^P(s) = \zeta_{\scB, \cO_n}(q^{-s}),
    \end{equation}
    where $\zeta_{\scB, \cO_n}(X)$ is the generating function from the vertex associated to $\cO_n$ for the corresponding impacted building. 
\end{theorem}
\begin{proof}
    We begin by noticing that a given \textit{potential} vertex/type-coordinate $(x, \omega)$ singles out at most one ideal. Indeed, Lemma \ref{lem: equal type an vertex implies equal} implies this ideal is unique if it exists. We conclude that summing over (principal, rank 2) ideals is the same as summing over vertex-type coordinates. Let us denote by $VT$ the set vertex type coordinates that do occur. 

    Our proof consists on rewriting in two ways the sum
    \begin{equation*}
        \displaystyle\sum_{(x, \omega)\in VT} q^{-c(\omega)s}.
    \end{equation*}
    This is an integral over a discrete set where the point $(x, \omega)$ has  measure $q^{-c(\omega)s}$ for a fixed positive $q$ and a fixed positive real $s$ large enough such that both series converge. If we can prove the equality in the set of real numbers, the functions are equal also for the complex values where both are defined due to uniqueness of analytic continuation. In particular, we will perform iterated integration in any order. 

    Let us first integrate over the vertex set. For each $x\in\scR_n$ the inner integral will be over the coordinates of the the set of ideals equivalent to that vertex. Theorem \ref{thm:structure of types at x} shows that these coordinates are
    \begin{equation*}
        (x, \omega(x)), (x, \omega(x) + \vec{e}), (x, \omega(x) + 2\vec{e}),...
    \end{equation*}
    In turn, the ideals specified by these coordinates contribute
    \begin{equation*}
        c(\omega(x)), c(\omega(x)) + 2, c(\omega(x)) + 4, \cdots
    \end{equation*}
    Indeed, the type contribution for a given possible type $\omega$ is
    \begin{equation*}
        c(\omega_1,...,\omega_g) = f_1\omega_1 + ... + f_g\omega_g,
    \end{equation*}
    where $f_1,..., f_g$ are the inertia degrees of the extensions $L_1,..., L_g$. In particular, because $e_if_i = [L_i: K]$, we have
    \begin{equation*}
        c(\vec{e}) = f_1e_1 + \cdots + f_ge_g = 2.
    \end{equation*}
    As we can see, because the type contribution is an additive function, our claim follows. Finally, using Lemma \ref{lem: contribution is distance}, we get these contributions are
    \begin{equation*}
        d, d + 2, d + 4, \cdots
    \end{equation*}
    where $d = d(x, \cO_n)$. We thus conclude that the iterated integral, obtained by integrating first over the vertex set, is
    \begin{equation}\label{eq: integrating first vertex}
        \displaystyle\sum_{x\in\scR_n}\left( q^{-d(x, \cO_n)} + q^{-(d(x, \cO_n) + 2)} + q^{-(d(x, \cO_n) + 4)} + \cdots\right).
    \end{equation}
    The generating function from the vertex $\cO_n$ is defined as
    \begin{equation*}
        \displaystyle\sum_{d = 0}^{\infty}\left\{x\in \scR_{n} \mid x \mbox{ can be reached by a path of length } d \mbox{ from } \cO_n  \right\}X^d
    \end{equation*}
    If $x$ can be reached by a path of length $d$ then by adding an \say{unnecessary} jump to a neighboring vertex and back (just as in Proposition~\ref{prop: geodesic non geodesic}), we see it can also be reached with a path of length $d + 2$. Thus, if $x\in\scR_n$ is a vertex at exact distance $d$ from $\cO_n$ then its contribution to the right hand side is
    \begin{equation*}
        X^d + X^{d + 2} + X^{d + 4} + \dots...,
    \end{equation*}
    which upon evaluation at $q^{-s}$ becomes the value \eqref{eq: integrating first vertex}. We conclude that
    \begin{equation*}
        \displaystyle\sum_{(x, \omega)\in VT} q^{-c(\omega)} = \zeta_{\scB, \cO_n}(q^{-s}).
    \end{equation*}
    We will now integrate first over the types. Doing so transforms the integral into
    \begin{equation*}
        \displaystyle\sum_{\omega}\displaystyle\sum_{\varepsilon(I)=\omega}q^{-c(\varepsilon(I))s}.
    \end{equation*}
    We know that the index $[\cO_n:I]$ only depends on the type. Concretely, if $c(\omega)$ is the \textit{type contribution}, then
    \begin{equation*}
        [\cO_n:I] = q^{-c(\varepsilon(I))}.
    \end{equation*}
    That is,
    \begin{equation*}
        \displaystyle\sum_{\omega}\displaystyle\sum_{\varepsilon(I)=\omega}q^{-c(\varepsilon(I))s} =\displaystyle\sum_{\omega}\displaystyle\sum_{\varepsilon(I)=\omega}\dfrac{1}{[\cO_n: I]^s}
    \end{equation*}
    As we have said, summing over $VT$ is the same as summing over the ideals, we conclude that the above integral is
    \begin{equation*}
        \displaystyle\sum_{I}\dfrac{1}{[\cO_n: I]^s}.
    \end{equation*}
    We can do this because we have eliminated the dependence on the type from the integrand. We recognize this as $\zeta_{\cO_n}^P(s)$. This concludes the proof.
\end{proof}



\bibliographystyle{acm}
\bibliography{Bibliography}

\end{document}